\documentclass[11pt]{amsart}

\usepackage{amsmath}
\usepackage{amssymb}
\usepackage{amscd}
\usepackage{amsthm}
\usepackage{hyperref}
\usepackage{colortbl}

\usepackage[all]{xy}

\newtheorem{thm}{Theorem}[section]
\newtheorem{lem}[thm]{Lemma}
\newtheorem{cor}[thm]{Corollary}
\newtheorem{conj}[thm]{Conjecture}
\newtheorem{prop}[thm]{Proposition}
\newtheorem{question}[thm]{Question}

\newtheorem{fact}[thm]{Fact}
\newtheorem*{thmm}{Theorem}

\theoremstyle{remark}

\newtheorem{rem}[thm]{Remark}

\theoremstyle{definition}
\newtheorem{defn}[thm]{Definition}

\numberwithin{equation}{section}

\allowdisplaybreaks[4]

\begin{document}

\vfuzz0.5pc
\hfuzz0.5pc 

\newcommand{\rank}{\mathop{\mathrm{rank}}}
\newcommand{\codim}{\mathop{\mathrm{codim}}}
\newcommand{\Ord}{\mathop{\mathrm{Ord}}}
\newcommand{\Var}{\mathop{\mathrm{Var}}}
\newcommand{\Ext}{\mathop{\mathrm{Ext}}}
\newcommand{\EXT}{\mathop{{\mathcal E}\mathrm{xt}}}
\newcommand{\Pic}{\mathop{\mathrm{Pic}}}
\newcommand{\Spec}{\mathop{\mathrm{Spec}}}
\newcommand{\Jac}{\mathop{\mathrm{Jac}}}
\newcommand{\Div}{\mathop{\mathrm{Div}}}
\newcommand{\sgn}{\mathop{\mathrm{sgn}}}
\newcommand{\supp}{\mathop{\mathrm{supp}}}
\newcommand{\Hom}{\mathop{\mathrm{Hom}}}
\newcommand{\Sym}{\mathop{\mathrm{Sym}}}
\newcommand{\nilrad}{\mathop{\mathrm{nilrad}}}
\newcommand{\Ann}{\mathop{\mathrm{Ann}}}
\newcommand{\Proj}{\mathop{\mathrm{Proj}}}
\newcommand{\mult}{\mathop{\mathrm{mult}}}
\newcommand{\Bs}{\mathop{\mathrm{Bs}}}
\newcommand{\Span}{\mathop{\mathrm{Span}}}
\newcommand{\IM}{\mathop{\mathrm{Im}}}
\newcommand{\Hol}{\mathop{\mathrm{Hol}}}
\newcommand{\End}{\mathop{\mathrm{End}}}
\newcommand{\CH}{\mathop{\mathrm{CH}}}
\newcommand{\Exec}{\mathop{\mathrm{Exec}}}
\newcommand{\SPAN}{\mathop{\mathrm{span}}}
\newcommand{\birat}{\mathop{\mathrm{birat}}}
\newcommand{\cl}{\mathop{\mathrm{cl}}}
\newcommand{\rat}{\mathop{\mathrm{rat}}}
\newcommand{\Bir}{\mathop{\mathrm{Bir}}}
\newcommand{\aut}{\mathop{\mathrm{aut}}}
\newcommand{\Aut}{\mathop{\mathrm{Aut}}}
\newcommand{\eff}{\mathop{\mathrm{eff}}}
\newcommand{\nef}{\mathop{\mathrm{nef}}}
\newcommand{\amp}{\mathop{\mathrm{amp}}}
\newcommand{\DIV}{\mathop{\mathrm{Div}}}
\newcommand{\Bl}{\mathop{\mathrm{Bl}}}
\newcommand{\Cox}{\mathop{\mathrm{Cox}}}
\newcommand{\NE}{\mathop{\mathrm{NE}}}
\newcommand{\NM}{\mathop{\mathrm{NM}}}
\newcommand{\Mov}{\mathop{\mathrm{Mov}}}
\newcommand{\Conv}{\mathop{\mathrm{Conv}}}

\newcommand{\claimref}[1]{Claim \ref{#1}}
\newcommand{\thmref}[1]{Theorem \ref{#1}}
\newcommand{\propref}[1]{Proposition \ref{#1}}
\newcommand{\lemref}[1]{Lemma \ref{#1}}
\newcommand{\coref}[1]{Corollary \ref{#1}}
\newcommand{\remref}[1]{Remark \ref{#1}}
\newcommand{\conjref}[1]{Conjecture \ref{#1}}
\newcommand{\questionref}[1]{Question \ref{#1}}
\newcommand{\defnref}[1]{Definition \ref{#1}}
\newcommand{\secref}[1]{\S\ref{#1}}
\newcommand{\ssecref}[1]{\ref{#1}}
\newcommand{\sssecref}[1]{\ref{#1}}

\newcommand{\RED}{{\mathrm{red}}}
\newcommand{\tors}{{\mathrm{tors}}}
\newcommand{\eq}{\Leftrightarrow}

\newcommand{\mapright}[1]{\smash{\mathop{\longrightarrow}\limits^{#1}}}
\newcommand{\mapleft}[1]{\smash{\mathop{\longleftarrow}\limits^{#1}}}
\newcommand{\mapdown}[1]{\Big\downarrow\rlap{$\vcenter{\hbox{$\scriptstyle#1$}}$}}
\newcommand{\smapdown}[1]{\downarrow\rlap{$\vcenter{\hbox{$\scriptstyle#1$}}$}}

\newcommand{\A}{{\mathbb A}}
\newcommand{\I}{{\mathcal I}}
\newcommand{\J}{{\mathcal J}}
\newcommand{\CO}{{\mathcal O}}
\newcommand{\C}{{\mathcal C}}
\newcommand{\BC}{{\mathbb C}}
\newcommand{\BQ}{{\mathbb Q}}
\newcommand{\m}{{\mathcal M}}
\newcommand{\h}{{\mathcal H}}
\newcommand{\Z}{{\mathcal Z}}
\newcommand{\BZ}{{\mathbb Z}}
\newcommand{\W}{{\mathcal W}}
\newcommand{\Y}{{\mathcal Y}}
\newcommand{\T}{{\mathcal T}}
\newcommand{\BP}{{\mathbb P}}
\newcommand{\CP}{{\mathcal P}}
\newcommand{\G}{{\mathbb G}}
\newcommand{\BR}{{\mathbb R}}
\newcommand{\D}{{\mathcal D}}
\newcommand{\DD}{{\mathcal D}}
\newcommand{\LL}{{\mathcal L}}
\newcommand{\f}{{\mathcal F}}
\newcommand{\E}{{\mathcal E}}
\newcommand{\BN}{{\mathbb N}}
\newcommand{\N}{{\mathcal N}}
\newcommand{\K}{{\mathcal K}}
\newcommand{\R} {{\mathbb R}}
\newcommand{\PP} {{\mathbb P}}
\newcommand{\Pp} {{\mathbb P}}
\newcommand{\BF}{{\mathbb F}}
\newcommand{\closure}[1]{\overline{#1}}
\newcommand{\EQ}{\Leftrightarrow}
\newcommand{\imply}{\Rightarrow}
\newcommand{\isom}{\cong}
\newcommand{\embed}{\hookrightarrow}
\newcommand{\tensor}{\mathop{\otimes}}
\newcommand{\wt}[1]{{\widetilde{#1}}}
\newcommand{\ol}{\overline}
\newcommand{\ul}{\underline}

\newcommand{\bs}{{\backslash}}
\newcommand{\CS}{{\mathcal S}}
\newcommand{\CA}{{\mathcal A}}
\newcommand{\Q} {{\mathbb Q}}
\newcommand{\F} {{\mathcal F}}
\newcommand{\sing}{{\text{sing}}}
\newcommand{\U} {{\mathcal U}}
\newcommand{\B}{{\mathcal B}}
\newcommand{\X}{{\mathcal X}}
\newcommand{\CG}{{\mathcal G}}
\newcommand{\CQ}{{\mathcal Q}}

\newcommand{\ECS}[1]{E_{#1}(X)}
\newcommand{\CV}[2]{{\mathcal C}_{#1,#2}(X)}

\title[Rationality of Euler-Chow Series]{Rationality of Euler-Chow Series and Finite Generation
of Cox Rings}

\author{Xi Chen}
\address{Department of Mathematics and Statistics\\
University of Alberta\\
Edmonton, Alberta T6G 2G1, CANADA}
\email{xichen@math.ualberta.ca}
\thanks{Research of Chen was supported in part by a grant from the Natural Sciences and Engineering Research Council of Canada.}

\author{E. Javier Elizondo}
\address{Instituto de Matem\'aticas\\ Universidad Nacional Aut\'onoma de
 M\'exico\\Ciudad Universitaria\\M\'exico DF 04510, M\'exico}
\email{javier@math.unam.mx}
\thanks{Research of Elizondo was supported in part by CONACYT 101519 and DGAPA 107012.}

\author{Yanhong Yang}
\address{Department of Mathematics\\
Columbia University\\
Room 509, MC 4406, 2990 Broadway\\
New York, NY 10027, USA}
\email{yhyang@math.columbia.edu}

\date{July 11, 2015}

\begin{center}To Blaine Lawson on the occasion of his 70th birthday.\end{center}

\begin{abstract}
We study the rationality of the Euler-Chow series $E^1(X)$ of codimension one cycles on a projective variety $X$ and its relation with the effective cone and the Cox ring of $X$. Among other things, we prove that $E^1(X)$ is transcendental if the cone $\NE^1(X)$ of pseudo-effective divisors on $X$ has infinitely many extremal rays generated by effective divisors.
On the other hand, we give examples showing that the converse fails.
In addition, we give an example where $E^1(X)$ is rational and
$\Cox(X)$ is infinitely generated. Finally, we compute $E^1(X)$ for Del Pezzo surfaces $X$.
\end{abstract}

\maketitle

\section{Introduction}\label{SECINTRO}

\subsection{Statement of results}

Zariski was the first to propose studying the dimensions of linear
systems in $\Pp^2$ passing through a fixed number of points in general
position with a given multiplicity. There has been a lot of work in
that direction since then. We can say that in general it is an
important and interesting problem to compute the dimensions of
different linear systems. It turns out to be a very hard problem as we will see that the corresponding generating function, called {\em Euler-Chow series}, is transcendental over a polynomial ring, if we have $9$ or more general points in $\PP^2$.

To answer Zariski's question, or more generally, to compute the dimension
$h^0(D)$ of every divisor $D$ on a smooth complex projective variety $X$ of dimension $n$, we introduce the generating function of $h^0(D)$ as
\begin{equation}\label{EC-1}
\ECS{n-1} = E^1(X) = \sum_{D \in \Pic(X)} h^0(D) \cdot t^D
\end{equation}
where $t$ is a formal variable and $E_{n-1}(X)$, written alternatively as $E^1(X)$, is called the $(n-1)$-dimensional Euler-Chow series of $X$.
For simplicity, we assume that the Picard group $\Pic(X)$ of $X$ is finitely generated, i.e., $H^1(\CO_X) = 0$. 

The precise definition of $E_\bullet(X)$ will be given in \secref{SECPREM}.

In general, $E^1(X)$ is very hard to compute. In the few cases where
it is known, including toric varieties and the blow-ups of $\PP^2$ at points lying on a line (cf. \cite{E} and \cite{EK2}), $E^1(X)$ turns out to be a rational function in
$\BQ(t_1, t_2, ..., t_m)$. Thus, we may ask the natural questions: Is
$E^1(X)$ always rational? If it is not, under what conditions is $E^1(X)$
rational or irrational? One of the main purposes of this paper is to address these two questions.

It is suspected that the rationality of $E^1(X)$ has much to do with the {\em Cox ring} of the variety. In the above-mentioned cases of toric varieties and the blow-ups of $\PP^2$ at points lying on a line, the Cox ring
or the total coordinate ring (cf. \cite{C} and \cite{EKW})
\begin{equation}\label{E100}
\Cox(X) = \bigoplus_{D\in \Pic(X)} H^0(D)
\end{equation}
is noetherian, i.e., finitely generated. We have the following simple fact:

\begin{fact}\label{FACT001}
Let $X$ be a smooth projective variety whose
Picard group is a free abelian group.
If $\Cox(X)$ is finitely generated, then
$E^1(X)$ is rational.
\end{fact}

The first purpose of our paper is to further investigate the relation between the rationality of $E^1(X)$ and the finite generation of $\Cox(X)$. Given the above fact, it is natural to ask whether the converse of the above statement holds:

\begin{question}\label{Q001}
Under the same hypothesis as above, does rationality of $E^1(X)$ imply
finite generation of $\Cox(X)$?
\end{question}

The answer to this question is negative. Our first result is the following counterexample.

\begin{thm}\label{THM000}
Let $S$ be a smooth quartic surface in $\PP^3$ and $X = \Bl_p S$ be the
blow-up of $S$ at a point $p\in S$. Then $\Cox(X)$ is not finitely generated and $E_1(X)$
is rational for $(S,p)$ very general.
\end{thm}

It has been brought to our attention that this example has already appeared in the work of Artebani and Laface \cite{A-L}.
Our argument for the infinite generation of $\Cox(X)$ is identical to theirs. So we do not pretend any originality in these parts.
We have kept our argument for the readers' convenience. However, our proof for the rationality of $E_1(X)$ is new, to the best of our
knowledge. 

Thus, infinite generation of $\Cox(X)$ does not guarantee
irrationality of $E^1(X)$. Next, we try to find some sufficient conditions
for $E^1(X)$ to be irrational. Our first example is the blow-up $X$ of $\PP^2$ at $9$ or more points in very general position,
corresponding to Zariski's problem mentioned at the very beginning.
This is probably the ``simplest'' surface whose Cox ring is not finitely generated.
It has been suspected that its Euler-Chow series is not rational for some time.
But the irrationality of $\ECS{1}$ has not been established until very recently.
Shun-ichi Kimura notified us that there was a paper in
preparation where it was proved that $\ECS{1}$ is irrational \cite{KKT}.

They based their proof on the well-known fact that the cone $\NE^1(X)$ of pseudo-effective divisors on $X$ is not a rational polyhedral cone
(actually not even a polyhedral cone) for such $X$ (see \secref{SECPREM} for further discussion of their theorem and proof).

We generalize this result in two theorems
\ref{THM001} and \ref{THM002}. The proofs of these two theorems are
completely different from that of their theorem.
In fact, their result does not apply to one of our examples, as we will see.
Moreover, we go one step further to show that $E^1(X)$ is transcendental
in a natural sense.

\begin{thm}\label{THM001} 
For every  pair of integers $p>1$ and $r>2$, let $q_0(r,p)$ be the minimal positive integer $q$ satisfying
\begin{equation}
q>r \text{ and }
\frac{1}{p}+\frac{1}{r}+\frac{1}{q-r}\leq 1.
\end{equation}
Then $E^1(X)$ is transcendental in the following cases: 
\begin{enumerate}
\item $X$ is the blow-up of $(\mathbb{P}^{r-1})^{p-1}$ at $\Lambda$, where $r>2$, $p>1$, $\Lambda$  is a finite set of points in $(\mathbb{P}^{r-1})^{p-1}$ and contains $q_0(r,p)$ points  in very general position.
		
\item $X$ is the blow-up of the product $\mathbb{P}^{r_1-1}\times\cdots\times\mathbb{P}^{r_{p-1}-1}$ at a finite set $\Lambda$, where $p>1$, $\Lambda$ lies on a  linear subspace $(\mathbb{P}^{r_0-1})^{p-1}$ with $2< r_0\leq \text{min}_{i=1}^{p-1}(r_i)$ and contains  $ q_0(r_0,p) $   points in very general position  as points of $(\mathbb{P}^{r_0-1})^{p-1}$. 
\end{enumerate}
\end{thm}

\begin{thm}
\label{THM002}
$E^1(X)$ is transcendental in the following cases:
\begin{enumerate}
\item $X$ is the blow-up of $\mathbb{P}^2$ at a finite set $\Lambda$,
where $\Lambda$ contains the intersection of two very general cubic curves.

\item $X$ is the blow-up of $\mathbb{P}^3$ at a finite set $\Lambda$,
where $\Lambda$ contains the intersection of three very general quadrics.

\item $X$ is the blow-up of  $\mathbb{P}^r$ at a finite set $\Lambda$,
where $\Lambda$ lies on a linear subspace $\mathbb{P}^2\subset \mathbb{P}^r$
and contains the intersection of two very general cubics.

\item $X$ is the blow-up of  $\mathbb{P}^r$ at a finite set $\Lambda$,
where $\Lambda$ lies on a linear subspace $\mathbb{P}^3\subset \mathbb{P}^r$
and contains the intersection of three very general quadrics.
\end{enumerate}
\end{thm}

In all these cases, the Cox rings are known to be infinitely generated:
Theorem \ref{THM001} is essentially a revisit of Mukai's famous counterexamples to Hilbert's 14th problem \cite{Mu1}; in Theorem \ref{THM002}, (1) is again due to S. Mukai \cite{Mu1}, (2) is a variation of $(1)$ due to A. Prendergast-Smith \cite{P} and (3) and (4) are basically due to B. Hassett and Y. Tschinkel \cite[Example 1.8]{H-T}. Hassett-Tschinkel's example is of special interest to us. The blow-up $X$ of $\PP^3$ at finitely many points $\Lambda$ lying on a plane $P$ has rational polyhedral $\NE^1(X)$,
which is obviously generated by the proper transform of $P$ and the exceptional divisors since a hypersurface of degree $d$ in $\PP^N$ has multiplicity at most $d$ at a point. On the other hand, the cone $\NM^1(X)$ of nef divisors on $X$ is not polyhedral under the hypothesis of Theorem \ref{THM002}, since there are infinitely many $(-1)$-curves on the proper transform of $P$ and hence the dual cone $\NE_1(X)$ of $\NM^1(X)$ is not polyhedral (see \cite{H-T}).
Thus, this gives us a smooth projective variety $X$ with rational polyhedral $\NE^1(X)$, non polyhedral $\NM^1(X)$ and hence infinitely generated
Cox ring $\Cox(X)$ (see \secref{SECPREM} for detailed discussions of Cox rings and Mori dream spaces). The theorem of Kimura-Kuroda-Takahashi cannot be directly applied here.

Our proofs of these two theorems depend on some general algebraic and geometric criteria
for $E^1(X)$ to be transcendental, given in Proposition \ref{PROP006} and Corollaries \ref{Transcendence 1} and \ref{PROP203}. These criteria can be used to show the transcendence of series other than
$E^1(X)$, e.g., the generating functions of Gromov-Witten invariants.

The second purpose of this paper is to compute $\ECS{1}$ for Del Pezzo surfaces.
In particular we can solve the problem posed by Zariski for all multiplicities when the number of points is less than $9$. Although it is known that $\ECS{1}$ is rational for Del Pezzo surfaces, it is only computed for $X$ the blow-up of $\PP^2$ up to 3 points, as these are toric varieties and they were computed in \cite{E}. Here we will try to develop a recursive formula for $\ECS{1}$ when $X$ is the blow-up of $\PP^2$ at $r\le 8$ general points and carry
out the computation for $r\le 4$. This computation also involves quadratic transforms,
which feature prominently in our proof of Theorem \ref{THM001}. In general, the behavior of $E^1(X)$ of a smooth projective variety $X$ under the blow-ups of $X$ at points is not well understood. We hope to be able to understand this behavior better with the computations carried out here.

\subsection{Outline of the paper}


The paper is organized as follows:  
In \secref{SECPREM} we provide more background materials on Euler-Chow series, Cox rings and Mori dream spaces.

In \secref{SECQUARTIC} we prove Theorem \ref{THM000} and state some
open questions on the loci of $(S,p)$ where $\Cox(X)$ is finitely
generated and/or $E_1(X)$ is rational for $X$ the blow-up of a quartic
$K3$ surface $S$ at a point $p\in S$.

In \secref{SECTECS} we start with some criteria for a series to be transcendental.
In particular, we prove Corollary \ref{Transcendence 1} 
and \ref{PROP203} which give geometric criteria for the transcendence of Euler-Chow series. Then Theorems \ref{THM001} and \ref{THM002} are proved using these transcendence criteria.

Finally in \S\ref{SECDPS} we compute the Euler-Chow series
of Del Pezzo surfaces. 

\subsection*{Conventions}

We work exclusively over $\BC$. Throughout the paper, if $X$ is a
variety of dimension $n$, then $E_{n-1}(X)$ is also denoted by $E^1(X)$.

\subsection*{Acknowledgments}

All authors would like to thank the referee for his/her hard work and many very helpful corrections and suggestions.

\section{Preliminaries}\label{SECPREM}

\subsection{Euler-Chow series}\label{1.1}


We are interested in the class of invariants for
projective varieties arising from the Euler characteristics of their Chow
varieties. 
In the case of the blow up of $\Pp^2$ at a finite number of points,
the problem posed by Zariski merges with the topological invariants,
more precisely, with computing the Euler characteristics of Chow
varieties of this variety.

We start by introducing the Euler-Chow series in general and then we
see what form it takes in the particular case that we are interested
here. We can take any of the equivalence relations we have for cycles; here we
take homological equivalence. The reader can look at other cases in \cite{EK2}.
Among other things, we are interested in the case where $X$ is
the blow up of $\Pp^2$ at a finite number of points in general
position.


For these cases it is also worth saying that there is a relation between
the series and the Cox ring, as will be shown later in this
section. 





\begin{defn}
\label{DEF1}
Given a projective variety $X$ over $\BC$, let $\lambda$ be an element in its
$2p$-homology group $H_{2p} (X, \BZ)$. Consider the monoid $M$ in
$H_{2p}(X, \BZ)$ given by algebraic classes of effective cycles. We
consider $M$ as a multiplicative monoid by $t^a t^b = t^{a+b}$ for
$a, b \in M$, where $t$ is a formal variable and $a\to t^a$ turns addition
in $M$ to multiplication.

The $p$-dimensional Euler-Chow Series of $X$ is defined as (cf. \cite{E})
$$
\ECS{p} \,= \, \sum_{\lambda \in M} \, \chi \big(\CV{p}{\lambda}\big)
\cdot t^{\lambda} \,\,\, \in \, \, \, \BZ [[M]]
$$
where $\CV{p}{\lambda}$ is the Chow variety parametrizing effective algebraic $p$-cycles homologous
to $\lambda$, $\chi \big(\CV{p}{\lambda}\big)$ denotes its Euler characteristic,
and $\BZ[[M]]$ is the ring of functions from $M$ to $\BZ$ with the convolution product.
For $f\in\BZ[[M]]$ with $f(\lambda) = a_\lambda$, we write
$f=\sum_{\lambda \in M} a_\lambda \cdot t^{\lambda}$. Then    
$$
\BZ[[M]] = \left\{ \sum_{\lambda \in M}   a_\lambda \cdot t^{\lambda} \, |
\, a_\lambda \in \BZ \right\},
$$
where the product on $\BZ[[M]]$ is the convolution: if
$f=\sum_{\lambda\in M} a_{\lambda} \cdot t^{\lambda}$ and
$g=\sum_{\gamma\in M} b_{\gamma} \cdot t^{\gamma}$,
then $f\cdot g= \sum_{\delta} \left( \sum_{\lambda +
\gamma = \delta} a_{\lambda} \cdot b_\gamma\right) \cdot  t^{\delta}$,
which is well-defined since the product operation
$\times :M\times M \rightarrow M$ has finite fibres due to the projectivity of $X$.
We denote by $\BZ[M]$
the ring contained in $\BZ[[M]]$ given by the elements with only a
finite number of $a_\lambda$ not zero.
Equivalently, $\BZ[M]$ is the monoid ring associated to $M$. The rationality
and algebraicity of $f\in \BZ[[M]]$ are defined in the following way.

We say that $f\in\BZ[[M]]$ is
{\em rational} if there are two elements $g,h$ in $\BZ[M]$,
not both zero, such that $g\cdot
f = h$. Similarly, we say that $f\in \BZ[[M]]$ is {\em algebraic} if there exist
$a_0, a_1, ..., a_d\in \BZ[M]$, not all zero, such that
$$
a_0 + a_1 f + ... + a_d f^d = 0.
$$
If $f$ is not rational or algebraic, we call $f$ {\em irrational}
or {\em transcendental}.
\end{defn}

Alternatively, we can define $\BZ[[M]]$ to be the inverse limit
of $\BZ[M]/I_d$ for a descending chain of ideals $I_d$ of $\BZ[M]$.
Due to the projectivity of $X$, we have a monoid homomorphism
$\deg: M\to \BN$
defined by $\deg \lambda = \lambda . A^p$ for a fixed ample divisor $A$ on $X$.
The standard argument using Hilbert schemes or Chow varieties shows that
\begin{equation}\label{E250}
\big| \{ \lambda\in M: \deg \lambda \le d \} \big| < \infty
\end{equation}
for all $d\in \BR$.
Therefore, 
$\BZ[M]/I_d$ is a finitely generated module over $\BZ$ for
$$
I_d = \left\{ \sum a_\lambda t^\lambda \in \BZ[M]: \deg \lambda \ge d
\text{ if } a_\lambda\ne 0 \right\}.
$$
Here $I_d$ satisfy $\cap I_d = \{0\}$ and $\BZ[[M]]$ can be alternatively defined by
$$
\BZ[[M]] = \varprojlim_{d} \BZ[M] / I_d.
$$
Note that this definition of $\BZ[[M]]$ is independent of the choice of the polarization $A$
of $X$ because, given two ample divisors $A$ and $B$ on $X$ with
$\deg_A \lambda = \lambda . A^p$ and $\deg_B \lambda = \lambda . B^p$, 
there exist positive constants $c_1$ and $c_2$ such that
$c_1 \deg_A (\lambda) \le \deg_B (\lambda) \le c_2 \deg_A(\lambda)$ for 
all $\lambda\in M$. Also note that the projectivity of $X$ is essential
in the definition of $\BZ[[M]]$.

The Euler-Chow series $E_\bullet(X)$ has been computed for some varieties $X$.
For example, when $X$
is a toric variety or the blow-up of $\PP^2$ at points lying on a line,
it has turned out to be rational. Once we have a rational function $E_\bullet(X)$ as the generating function, we can compute, with a little algebra, its
coefficients that are the Euler characteristic of Chow
varieties of $X$.

Let us start with a simple example, the case of zero-dimensional cycles. As
before, let $X$ be a projective variety. Since we are considering
elements $\lambda$ in the zeroth homology group, we have that
$\lambda$ must be equal to a nonnegative integer, and it is well
known that $\CV{0}{d}$ is isomorphic to the $d$-fold symmetric product
$SP^d (X)$. In this case, the $0$-dimensional Euler-Chow Series is
$$
\ECS{0} \, = \, \sum_{d=0}^{\infty} \, \chi\big( SP^d (X) \big) \cdot t^d
$$
and a result of Macdonald \cite{Mac} shows that $E_0 (X)$ is given by
rational function $E_0 (X) \, =\, (1/(1-t))^{\chi(X)}$.

Another familiar instance arises in the case of divisors.
Let $X$ be a smooth projective variety of dimension $n$ satisfying $H^1(\CO_X) = 0$.
Then the Picard group $\Pic(X) = H^1(\CO_X^\times)$ of $X$ is a subgroup of $H^2(X, \BZ)$ and hence finitely generated.
Let ${\Div}_+ (X)$ be the space of effective divisors on $X$
and let 
\begin{equation}\label{E000}
M = M_X = {\Div}_+(X) / \sim
\end{equation}
be the monoid of effective divisors modulo linear
equivalence. Observe that
\begin{enumerate}
\item[A.-]
Given $L\in \Pic(X)$, then $\dim H^0(X, L) \not= 0$ if and
only if $L =\CO(D)$ for some effective divisor $D$.
\item[B.-]
Under the given hypothesis, homological and linear equivalence
coincide, and two effective divisors $D$ and $D^\prime$ are
homologically equivalent if and only if they are in the same linear
system. Therefore, $\CV{n-1}{\lambda} = \PP H^0(X, \CO(D))$ and hence
$\chi(\CV{n-1}{\lambda}) = h^0(X, \CO(D))$ with $[D] = \lambda$.
Namely, the Euler characteristics of Chow
varieties of divisors are the dimensions of complete linear systems.
\end{enumerate}

Therefore, the $(n-1)$-dimensional Euler-Chow series
$E_{n-1}(X) = E^1(X)$ is exactly defined by \eqref{EC-1}.

\subsection{Rationality of Euler-Chow series}\label{1.2}

From now on, we focus on $E^1(X)$ exclusively and
$M = M_X$ always refers to the monoid of effective divisors on $X$.
In general, $E^1(X)$ is very hard to compute.
It is only computed for some very special varieties $X$
(cf. \cite{ELF}). In all the known cases including abelian varieties,
toric varieties and the blow-ups of $\PP^2$ at points lying on a line, 
$E^1(X)$ turns out to be a rational function. 
Rationality of Euler-Chow series has been studied in
\cite{EK1}, \cite{ES} and \cite{EK2}. The case of abelian varieties
was worked out in \cite{EH}, for toric varieties in \cite{E}
and for the blow-ups of $\PP^2$ at points lying on a line in \cite{EK2}.
As pointed out in \secref{SECINTRO}, the rationality of $E^1(X)$
is closely related to the finite generation of the Cox ring $\Cox(X)$.
However, Theorem \ref{THM000} shows that these two notions are not equivalent.

To prove Theorem \ref{THM000}, we need the result
of Y. Hu and S. Keel that characterizes a variety with finitely generated Cox ring
geometrically \cite[Proposition 2.9]{H-K}:

\begin{thmm}[Hu-Keel]
Let $X$ be a smooth projective variety whose Picard group is a
finitely generated free abelian group. Then
$\Cox(X)$ is finitely generated if and only if $X$ is a Mori dream space (MDS).
\end{thmm}

In order to explain what a MDS is, we need to introduce a few basic
concepts in birational geometry:

\begin{defn}\label{DEFN002}
Let $\NE_k(X)\subset H^{2n-2k}(X, \BR)$
be the cone of pseudo-effective algebraic cycles of dimension $k$ on $X$.
That is, it is the smallest closed real cone in $H^{2n-2k}(X, \BR)$ containing all the
effective algebraic cycles of dimension $k$. For convenience, we write
$\NE^k(X) = \NE_{n-k}(X)$. So $\NE^1(X)$ is the smallest closed real cone
containing all the effective divisors in $H^2(X, \BR)$. Namely,
it is the closure $\overline{\Conv(M_X)}$ of the convex hull of $M_X$ in $H^2(X,\BR)$.
It is usually called the {\em cone of (pseudo-)effective divisors} or
{\em effective cone of divisors} on $X$.

The nef cone $\NM^k(X) = \NE_k(X)^\vee$ is the dual cone of $\NE_k(X)$ in
the subspace of $H^{2k}(X, \BR)$ spanned by the algebraic cycles of codimension $k$.
In particular, $\NM^1(X)\subset H^2(X, \BR)$
is the smallest closed real cone containing all the numerically effective
(nef) divisors and it is a subcone of $\NE^1(X)$ by Kleiman's criterion.

A divisor $D$ is {\em semi-ample}
if the complete linear series $|mD|$ is base point free
for some $m \in \BZ^+$.
 
Let $\LL$ be a linear system on a smooth projective variety.
For a general member $D\in \LL$, we
can write
\begin{equation}\label{E076}
D = D_f + D_\mu
\end{equation}
as a sum of two effective divisors, where $D_f$ is the {\em fixed part} of $\LL$
satisfying $D_f\subset D'$ for every $D'\in \LL$ and $D_\mu$ is the {\em moving part} of $\LL$
satisfying $\dim(D_\mu \cap D') < \dim X - 1$ for $D'\in \LL$ general. 
The fixed part $D_f$ and the moving part $D_\mu = D - D_f$ of a divisor $D$ are those of the complete linear system $|D|$.
And we call a divisor $D$
{\em movable} if $D_f = 0$.
The smallest closed real cone in $H^2(X, \BR)$ containing all the movable divisors is called
{\em the moving cone} of $X$ and denoted by $\Mov(X)$.
\end{defn}

\begin{defn}\label{DEFN003}
A normal $\BQ$-factorial projective variety $X$ is a Mori Dream Space (MDS) if
\begin{enumerate}
\item[MD1.]
Every nef divisor on $X$ is semi-ample and
the nef cone $\NM^1(X)$ is generated by finitely many semi-ample divisors.
\item[MD2.]
There exists a finite collection of birational maps $f_i: X_i\dashrightarrow X$
such that $f_i$ is an isomorphism in codimension one, $X_i$ is $\BQ$-factorial,
$\NM^1(X_i)$ is generated by finitely many semi-ample divisors and the moving cone $\Mov(X) = \cup (f_i)_* \NM^1(X_i)$.
\end{enumerate}
\end{defn}

Note that we do not explicitly assume that $\Pic_\BQ(X) = N^1(X)$, where $N^1(X)$ is the Neron-Severi group of $X$. But it is implied by the hypothesis that every nef divisor on $X$ and $X_i$ is semi-ample.

All toric and Fano varieties are MDS. So their Euler-Chow series are
rational. In the case of toric varieties, explicit computation was
made in \cite{E}. Later in this paper, we will compute $E^1(X)$ for $X$ the blow-up of $\PP^2$ at $r\le 8$ general points, which are Del Pezzo surfaces and special cases
of Fano varieties.

For $X$ to be a MDS, we see that its nef cone $\NM^1(X)$, a priori, has to be
rational polyhedral. Another necessary condition for $X$ to be a MDS is that
$\NE^1(X)$ is also rational polyhedral. This is clear if we apply Hu-Keel's theorem since
$\NE^1(X)$ is obviously rational polyhedral if $\Cox(X)$ is finitely generated.
We can also see this directly from MD1 and MD2:

\begin{prop}\label{PROPMOVC}
For a normal $\BQ$-factorial projective variety $X$ satisfying MD1 and MD2,
$\NE^1(X)$ is rational polyhedral.
\end{prop}

\begin{proof}
Fixing a semi-ample divisor $F$,
we claim that there are only finitely many integral (i.e. reduced and irreducible or prime) divisors $D$ satisfying
\begin{equation}\label{EMOVC001}
m F - D \in \text{NE}^1(X) \text{ for } m >> 1
\text{ and }
D + t F \not\in \Mov(X) \text{ for all } t
\end{equation}
on a normal $\BQ$-factorial projective variety $X$.

It is enough to prove this for $X$ smooth since a pair $(D, F)$ satisfies \eqref{EMOVC001} on $X$ only if $(\widehat{D}, f^* F)$ has
the same property on $\widehat{X}$ for a desingularization
$f: \widehat{X} \to X$ of $X$ with $\widehat{D}$ the proper transform of $D$ under $f$.

Let $\phi: X \to Y\subset \PP H^0(mF)^\vee$ be the map given by $|mF|$ such
that $\phi$ is surjective and $\phi^* G = mF$ for some $m\in \BZ^+$ and $G = \CO_Y(1)$. Using Stein factorization,
we may assume that $Y$ is normal and $\phi$ has connected fibers.
Suppose that $l m F - D\in \NE^1(X)$, i.e., $l \phi^* G - D$ is pseudo-effective
for $l >> 1$. Then we must have $D \cap X_y = \emptyset$ for a general fiber $X_y = \phi^{-1}(y)$ of
$\phi$. It follows
that $E = \phi(D)$ is a proper subvariety of $Y$.
There are only finitely many integral divisors $D$ such that $E = \phi(D)$ has codimension $\ge 2$ in $Y$ or $E = \phi(D')$ for some integral divisor $D'\ne D$.
Let us assume that $E$ is an integral divisor on $Y$
and $E \ne \phi(D')$ for all integral divisors $D'\ne D$.
Namely, there exists a closed subvariety $Z\subset Y$ of
codimension $\codim_Y Z \ge 2$ such that
\begin{equation}\label{EMOVC002}
D \cap \phi^{-1}(U) = \phi^{-1}(E\cap U)
\end{equation}
for $U= Y\backslash Z$. Obviously, we may choose $Z$ such that $U$ is smooth.
Now we are going to show that $D + tF \in \Mov(X)$ for $t >> 1$.

Let $b = \dim Y$. If $b = 0$, $F = 0$ and this is impossible.
If $b = 1$,
$D$ is supported on a fiber of $\phi$ and $D = \phi^* L$ for some $\BQ$-divisor $L$
on $Y$; hence $\mu(D + N \phi^* G)$ is movable for some $\mu\in \BZ^+$ and $N >> 1$.
So we assume that $b\ge 2$.

Let $f: \widehat{X} \to X$ and $g: \widehat{Y} \to Y$ be proper
birational morphisms with the commutative diagram
\[
\xymatrix{
\widehat{X} \ar[r]^{\widehat{\phi}} \ar[d]_{f} & \widehat{Y} \ar[d]^g\\
X \ar[r]^\phi & Y
}
\]
where $\widehat{X}$ and $\widehat{Y}$ are smooth
and $f$ and $g$ are isomorphisms over $\phi^{-1}(U)$ and $U$,
respectively.
Let $\widehat{E}\subset \widehat{Y}$ be the proper transform of $E$ under $g$.
On $\widehat{Y}$, we have the inequality
\[
\begin{split}
h^0(\widehat{E}, \widehat{E} + N g^* G) &\ge h^0(\widehat{Y}, \widehat{E} + N g^* G) - h^0(\widehat{Y}, N g^* G)\\
&\ge h^0(\widehat{E}, \widehat{E} + N g^* G) - h^1(\widehat{Y}, N g^* G).
\end{split}
\]
Using Leray spectral sequence, we have $h^1(\widehat{Y}, N g^* G) = O(N^{b-2})$. Thus,
\[
\begin{split}
h^0(\widehat{Y}, \widehat{E} + N g^* G) &= h^0(\widehat{Y}, N g^* G) + 
\frac{E.G^{b-1}}{(b-1)!} N^{b-1} +
O(N^{b-2})\\
&= h^0(Y, NG) + \frac{E.G^{b-1}}{(b-1)!} N^{b-1} +
O(N^{b-2})\\
&= h^0(X, NmF) + \frac{E.G^{b-1}}{(b-1)!} N^{b-1} +
O(N^{b-2}).
\end{split}
\]
Let $\widehat{D} \subset \widehat{X}$ be the proper transform of $D$ under $f$. Our
hypotheses on $D$ imply that 
\[
\widehat{\phi}^* \widehat{E} = \mu \widehat{D} + J
\]
for some $\mu\in \BZ^+$ and some effective divisor $J$ satisfying $\dim(g\circ \widehat{\phi}(J))
\le b - 2$, since $(g\circ\widehat{\phi})^{-1} (E\cap U) = \widehat{D}
\cap (g\circ \widehat{\phi})^{-1}(U)$ by \eqref{EMOVC002}. Therefore,
we have the estimate
\[
\begin{split}
h^0(\widehat{X}, \mu \widehat{D} + N f^*(mF))
&= h^0(\widehat{X}, \mu \widehat{D} + J + N f^*(mF)) + O(N^{b-2})\\
&= h^0(\widehat{Y}, \widehat{E} + N g^* G) + O(N^{b-2})\\
&= h^0(X, N m F) + \frac{E.G^{b-1}}{(b-1)!} N^{b-1} +
O(N^{b-2}).
\end{split}
\]
Similarly, working with the pullback $f^*(\mu D)$, we obtain
\[
\begin{split}
h^0(X, \mu D + NmF) &= h^0(\widehat{X}, \mu \widehat{D} + N f^*(mF)) + O(N^{b-2})
\\
&= h^0(X, N m F) + \frac{E.G^{b-1}}{(b-1)!} N^{b-1} +
O(N^{b-2}).
\end{split}
\]
This implies that $D + t F\in \Mov(X)$ for $t >> 1$.
Therefore, there are only finitely many
$D$ satisfying \eqref{EMOVC001}. Note that this holds for arbitrary $X$ and the hypotheses
MD1 and MD2 do not come into play.

Now let us assume that $X$ satisfies MD1 and MD2.
For $F\in \Mov(X)$, we use the notation $\Pi_F$ to denote the set
of integral divisors $D$ satisfying \eqref{EMOVC001}. Let us make the following two observations:
\begin{itemize}
\item
For all $F_1, F_2, ..., F_n\in \Mov(X)$ and $c_1, c_2, ..., c_n > 0$,
\[
\Pi_{c_1 F_1 + c_2 F_2 + ... + c_n F_n} = \Pi_{F_1 + F_2 + ... + F_n}
\]
and hence
\begin{equation}\label{EMOVC003}
\bigcup_{c_1, c_2,...,c_n\ge 0} \Pi_{c_1 F_1 + c_2 F_2 + ... + c_n F_n}
= \bigcup_{c_1, c_2,...,c_n\in \{0,1\}} \Pi_{c_1 F_1 + c_2 F_2 + ... + c_n F_n}.
\end{equation}
\item
Since $f_i: X_i\dashrightarrow X$ is an isomorphism in codimension one,
\[
{\Pic}_\BQ(X)\isom {\Pic}_\BQ(X_i), \Mov(X)\isom \Mov(X_i)
\text{ and }
{\NE}^1(X)\isom {\NE}^1(X_i).
\]
A pair $(D, F)$ of an integral divisor $D$ and a movable divisor
$F$ satisfies \eqref{EMOVC001} on $X$ if and only if $(f_i^* D, f_i^* F)$ satisfies
\eqref{EMOVC001} on $X_i$. Therefore,
\begin{equation}\label{EMOVC004}
\Pi_{f_i^* F} = f_i^* \Pi_{F} \text{ for all } F\in \Mov(X).
\end{equation}
\end{itemize}
Combining \eqref{EMOVC003}, \eqref{EMOVC004} and the fact that
$\Pi_F$ is finite for $F$ semi-ample, we conclude that the set
\[
\begin{split}
\Pi &= \bigcup_{F\in \Mov(X)} \Pi_F\\
&= \big\{ D \text{ integral divisor}: mF - D \in \text{NE}^1(X)
\text{ and } D + tF \not\in \Mov(X)
\\
&\quad\quad 
\text{ for some } F \in \Mov(X), m >> 1 \text{ and all } t\big\}
\end{split}
\]
is finite. Let $\Sigma$ be the set of integral divisors $D$ on $X$
satisfying that $D\not\in \Pi$ and $D\not\in \Mov(X)$.

For every finite subset $S\subset \Sigma$, we claim that
\begin{equation}\label{GSRINGEMOVC002}
C_S^\circ \cap \Mov(X) = \emptyset
\end{equation}
where $C_S$ is the cone generated by $S\subset H^2(X,\BR)$
and 
$$
C_S^\circ = \left\{ \sum_{D\in S} a_D D: a_D > 0\right\}
$$
is the interior of $C_S$. For every $F\in C_S^\circ\cap \Mov(X)$,
since $m F - D\in \NE^1(X)$ for all $D\in S$ and $m >> 1$, 
there exists $\varepsilon > 0$ such that $F + \varepsilon D\in \Mov(X)$ for all $D\in S$.
Thus, $C_S^\circ \cap \Mov(X)$ contains the (open) cone
$$
W_F = \left\{ 
\sum_{D\in S} (b_D F + \varepsilon b_D D): b_D > 0 
\right\}$$
generated by $F + \varepsilon D$ and $W_F$ is an open subset of $C_S^\circ$
containing $F$.
So $C_S^\circ \cap \Mov(X)$ is both open and closed
in $C_S^\circ$ and \eqref{GSRINGEMOVC002} follows.
This implies that
\begin{equation}\label{GSRINGEMOVC000}
C_\Sigma\cap \Mov(X) = \{0\}.
\end{equation}
We observe that every nonzero divisor $F\in C_S\cap C_T$ for two
disjoint subsets $S$ and $T$ of $\Sigma$ is movable. Therefore,
it follows from \eqref{GSRINGEMOVC000} that
\begin{equation}\label{GSRINGEMOVC001}
C_S \cap C_T = \{0\} \text{ for all } S, T\subset \Sigma
\text{ and } S\cap T = \emptyset.
\end{equation}

It is a well-known fact in convex geometry that every set $\Sigma$ of
$\ge n+2$ points in $\BR^n$ can be divided into two disjoint subsets
$S$ and $T$ such that the convex hulls of $S$ and $T$ have non-empty intersection: For
$n+2$ points $p_1, p_2, ..., p_{n+2}\in \BR^n$, we can find $a_1, a_2, ..., a_{n+2}\in \BR$,
not all zero, such that $\sum a_k = 0$ and $\sum a_k p_k = 0$; then we simply let
$S = \{ p_k: a_k \ge 0\}$ and $T = \{ p_k: a_k < 0\}$.

Consequently, $|\Sigma| \le h^2(X) + 1 < \infty$ and hence
\[
\text{NE}^1(X) = \Mov(X) + C_\Pi + C_\Sigma
\]
is rational polyhedral, since $\Mov(X)$ is rational polyhedral by MD1 and MD2.
\end{proof}

For a smooth projective surface $X$, $\NM^1(X)$ and $\NE^1(X)$ are dual to each other and 
every movable divisor on $X$ is nef. Therefore, MD1 is sufficient for surfaces to have
finitely generated Cox rings. That is, when $\dim X = 2$,
$\Cox(X)$ is finitely generated if and only if
its nef cone is rational polyhedral and every nef divisor on $X$ is semi-ample. Our
counterexample to \questionref{Q001} is exactly a smooth projective surface $X$
with rational polyhedral cones $\NE^1(X)$ and $\NM^1(X)$ and a nef divisor that is not
semi-ample.

Despite Theorem \ref{THM000}, we still expect that \questionref{Q001} holds true for
a certain class of varieties. We tentatively make the following conjecture:

\begin{conj}\label{CONJ001}
Let $X$ be a smooth rationally connected projective variety.
Then $E^1(X)$ is rational if and only if $\Cox(X)$ is finitely generated. 
\end{conj}

Note that $\Pic(X)$ is automatically finitely generated and free
if $X$ is a smooth rationally connected projective variety.
Otherwise, there is a torsion line bundle on $X$ giving rise to
a nontrivial \'etale morphism $f: Y \to X$. Obviously,
$Y$ is also rationally connected and hence $\chi(\CO_Y) = \chi(\CO_X) = 1$.
But $\chi(\CO_Y) = (\deg f) \chi(\CO_X)$, which is a contradiction.
Of course, this shows that a smooth rationally
connected projective variety is simply connected, which is a well-known fact.

So far we do not have much evidence supporting the conjecture. But in the examples
we have where $X$ is rational and $\Cox(X)$ is known to be infinitely generated,
as in Theorems \ref{THM001} and \ref{THM002},
we can always prove that $E^1(X)$ is irrational. And these examples are interesting in their own rights.

The simplest such example is the blow-up $X$ of $\PP^2$ at $9$ very general points. The irrationality of $E^1(X)$ is a consequence of
Theorem \ref{THM001}. As mentioned in \secref{SECINTRO}, this was also independently proved by Kimura-Kuroda-Takahashi.
Actually, they proved the following algebraic result \cite[Theorem 1.1]{KKT}:

\begin{thmm}[Kimura-Kuroda-Takahashi]
The cone associated to a series
$f(t) = \sum a_\lambda t^\lambda\in \BZ[[t_1, t_2, ..., t_m]]$,
i.e., the smallest closed real cone in $\BR^m$ containing $\{\lambda: a_\lambda \ne 0\}$,
is a rational polyhedron if $f(t)$ is rational.
Consequently, 
$\NE^1(X)$ is a rational polyhedron if $E^1(X)$ is rational for a smooth projective
variety $X$ with $\Pic(X)\isom \BZ^m$.
In particular, $E_1(X)$ is irrational for the blow-up $X$ of $\PP^2$ at $9$ or more
very general points.
\end{thmm}

Thus, for the case where $X$ is the blow-up of $\PP^2$ at $\Lambda$, we can say that $\ECS{1}$ is very hard, if not impossible, to compute if $\Lambda$ consists of $r\ge 9$ points in very general position. On the other hand, it should be pointed out that $\Bl_\Lambda\PP^2$ can still be a MDS if the points in $\Lambda$ are not in general position. For example, if $\Lambda$ consists of points lying on a line, $X = \Bl_\Lambda\PP^2$ is a MDS and $\ECS{1}$ has been computed by E. Javier Elizondo and  Shun-ichi Kimura in \cite{EK2} using its motivic version, the motivic Chow series.

\section{Blow-ups of Quartic $K3$}
\label{SECQUARTIC}

\subsection{Proof of Theorem \ref{THM000}}

Let $L$ be the hyperplane divisor on the smooth quartic surface $S\subset\mathbb{P}^3$  and $C\in |L|$ be the curve cut out by the tangent
plane of $S$ at $p$. Then $C$ is a quartic plane curve with exactly one node for $p\in S$
general. Let $\widehat{C}\subset X$ be the proper transform of $C$ under the blow-up $\pi: X\to S$.
Obviously, $\widehat{C} = \widehat{L} - 2E$, where $\widehat{L} = \pi^* L$
and $E\subset X$ is the exceptional divisor of $\pi$.

Since $\widehat{C}$ is irreducible and $\widehat{C}^2 = 0\ge 0$, $\widehat{C}$ is nef.
Indeed, it is
easy to see that $\NM^1(X)$ is generated by $\widehat{C}$ and $\widehat{L}$ and
$\NE^1(X)$ is generated by $\widehat{C}$ and $E$. We claim that $\widehat{C}$ is not
semi-ample. That is, $h^0(\CO_X(n \widehat{C})) = 1$
for all $n\in \BZ^+$. From the exact sequence
\begin{equation}\label{E200}
0 \xrightarrow{} H^0(\CO_X((n-1) \widehat{C}))
\xrightarrow{} H^0(\CO_X(n \widehat{C}))
\xrightarrow{} H^0(\CO_{\widehat{C}} (n \widehat{C})),
\end{equation}
we see that $h^0(\CO_X(n \widehat{C})) = 1$ as long as
\begin{equation}\label{E203}
H^0(\CO_{\widehat{C}} (n \widehat{C})) = 0
\end{equation}
for all $n\in \BZ^+$. Note that 
\begin{equation}\label{E201}
(E + \widehat{C})|_ {\widehat{C} }= (K_X + \widehat{C})|_{ \widehat{C}} = K_{\widehat{C}}
\end{equation}
in $\Pic(\widehat{C})$ by adjunction, where $K_X$ and $K_{\widehat{C}}$ are
the canonical divisors of $X$ and $\widehat{C}$, respectively. Therefore,
\begin{equation}\label{E202}
\widehat{C}|_{ \widehat{C}} = K_{\widehat{C}} - E|_{ \widehat{C}} = K_{\widehat{C}} - q_1 - q_2
\end{equation}
in $\Pic(\widehat{C})$,
where $q_1$ and $q_2$ are the two points on $\widehat{C}$ over $p$. Therefore,
\eqref{E203} holds as long as $K_{\widehat{C}} - q_1 - q_2$ is non-torsion
in $\Pic(\widehat{C})$. 

\begin{lem}\label{LEM200}
For a very general quartic $K3$ surface $S$ and a very general point $p\in S$,
$K_{\widehat{C}} - q_1 - q_2$ is non-torsion.
\end{lem}

\begin{proof}
We fix a plane $\Lambda\subset\PP^3$ and consider $W\subset |\CO_{\PP^3}(4)|$ consisting of
all quartic surfaces $S$ tangent to $\Lambda$. Obviously, we have a dominant rational
map $W \dashrightarrow V_{4,2}$ sending $S$ to $S\cap \Lambda$, where $V_{d,g}$ is the
Severi variety parametrizing nodal plane curves of degree $d$ and genus $g$.
And $V_{4,2}$ in turn
maps to the moduli space of genus $2$ curves with two unordered points via the map sending $C$ to $(\widehat{C}, q_1, q_2)$, where $\widehat{C}$ is the normalization
of $C$ and $q_1$ and $q_2$ are the two points on $\widehat{C}$ over the node $p\in C$.
It is easy to see that this map is dominant: For every smooth curve $\widehat{C}$ of
genus $2$ and two points $q_1$ and $q_2$ on $\widehat{C}$, we can map
$\widehat{C}$ to $\PP^2$ using $|K_{\widehat{C}} + q_1 + q_2|$; we can choose
three linearly independent sections $s_1, s_2$ and $s_3$ of $H^0(K_{\widehat{C}} + q_1 + q_2)$
with $s_1(q_i) = s_2(q_i) = 0$ for $i=1,2$; then the map $q\to (s_1(q), s_2(q), s_3(q))$
sends $q_1$ and $q_2$ to the same point $p = (0,0,1)$ and maps $\widehat{C}$ to
$C\in V_{4,2}$ with a node at $p$ if $q_1\ne q_2$.

In summary, we have dominant maps
\begin{equation}\label{E204}
W\dashrightarrow V_{4,2} \xrightarrow{} \m_{2,2} // \Sigma_2
\end{equation}
where $\m_{g,n}$ is the moduli space of genus $g$ curves with $n$ marked points and 
its quotient
by the symmetric group $\Sigma_n$ on the $n$ marked points is the moduli space of genus $g$
curves with $n$ unordered points.
Obviously, $K_{\widehat{C}} - q_1 - q_2$ is non-torsion for a very general point
$(\widehat{C}, q_1, q_2)$ of $\m_{2,2}$.
\end{proof}

Therefore, $\widehat{C}$ is nef and not semi-ample and $X$ is not a MDS. It follows that
$\Cox(X)$ is not finitely generated by the theorem of Hu-Keel. However, its Euler-Chow series
$E_1(X)$ can be explicitly computed as follows and it turns out to be rational.

We write
\begin{equation}\label{E205}
\begin{split}
E_1(X) &= \sum_{a,b\ge 0} h^0(a \widehat{C} + b E) t_1^a t_2^b
\\
&= \left(\sum_{a\ge b=0} + \sum_{b>a=0} + \sum_{b\ge 2a > 0} + \sum_{2a > b > 0} \right)
h^0(a \widehat{C} + b E) t_1^a t_2^b
\end{split}
\end{equation}
where $t_1 = t^{\widehat{C}}$ and $t_2 = t^E$.

We have proved that $h^0(a\widehat{C}) = 1$ for $a \ge 0$. Hence
\begin{equation}\label{E206}
h^0(a\widehat{C} + bE) = 1
\end{equation}
when $a = 0$ or $b = 0$. And it is trivial that
\begin{equation}\label{E207}
h^0(a\widehat{C} + bE) = h^0(a\widehat{L}) = 2a^2 + 2
\end{equation}
when $b\ge 2a > 0$.

When $2a > b > 0$, we have
\begin{equation}\label{E208}
\begin{split}
&\quad
h^0(a\widehat{C} + bE) - h^1(a\widehat{C} + bE) + h^2(a\widehat{C} + bE)\\
&= 
2ab - a - \frac{b(b-1)}2 + 2
\end{split}
\end{equation}
by Riemann-Roch. We have the vanishing
\begin{equation}\label{E209}
h^2(a\widehat{C} + bE) = h^0(- a\widehat{C} - (b-1) E) = 0
\end{equation}
since $(- a\widehat{C} - (b-1) E) \widehat{L} < 0$ as long as $a > 0$.
Also 
$$
h^1(a\widehat{C} + bE) = h^1(K_X + (a\widehat{C} + bE - K_X))
= h^1(K_X + (a\widehat{C} + (b-1)E)) = 0
$$
for $2a > b > 1$ since $a\widehat{C} + (b-1)E$
is ample in this case. When $2a > b = 1$, we have
\begin{equation}\label{E210}
\begin{split}
H^1(\CO_X((a-1)\widehat{C} + E))
&\xrightarrow{} H^1(\CO_X(a \widehat{C} + E))
\xrightarrow{} H^1(\CO_{\widehat{C}} (a \widehat{C} + E))\\
&\xrightarrow{} H^2(\CO_X((a-1)\widehat{C} + E)) \xrightarrow{} 0,
\end{split}
\end{equation}
where $h^1(\CO_{\widehat{C}} (a \widehat{C} + E)) 
= h^0(\CO_{\widehat{C}}((1-a) \widehat{C}))$. Note that
$h^0(\CO_{\widehat{C}}((1-a) \widehat{C})) = 0$ for $a \ne 1$ by Lemma \ref{LEM200}.
When $a = 1$, \eqref{E210} becomes
\begin{equation}\label{E211}
\xymatrix{0 \ar[r] & H^1(\CO_X(\widehat{C} + E)) \ar[r] &
H^1(K_{\widehat{C}}) \ar[r]\ar@{=}[d] &
H^2(\CO_X(E)) \ar[r]\ar@{=}[d] & 0\\
& & \BC & \BC}
\end{equation}
and hence $H^1(\widehat{C} + E) = 0$. Then $H^1(a \widehat{C} + E) = 0$ for all $a > 0$ by
induction using \eqref{E210}.
In conclusion, $h^1(a\widehat{C} + bE) = h^2(a\widehat{C} + bE) = 0$ and hence
\begin{equation}\label{E212}
h^0(a\widehat{C} + bE) = 2ab - a - \frac{b(b-1)}2 + 2
\end{equation}
when $2a > b > 0$.

\begin{rem}\label{REM200}
Even without Hu-Keel's theorem, we can directly see that $\Cox(X)$ is not finitely generated
by this computation. Setting $b =1$ in \eqref{E212}, we obtain
\begin{equation}\label{E224}
h^0(a\widehat{C} + E) = a + 2
\end{equation}
for all $a \ge 1$. It follows that the map
\begin{equation}\label{E225}
H^0(\widehat{C})\tensor H^0((a-1)\widehat{C} + E) \xrightarrow{} H^0(a\widehat{C} + E) 
\end{equation}
is not surjective and hence there exists an irreducible curve $D_a\in |a\widehat{C} + E|$
for each $a\ge 1$. The ideal generated by $\{ D_a: a\in \BZ^+\}\subset \Cox(X)$
is obviously not finitely generated since $D_a$ does not lie in the image of
$$
\sum_{k=1}^a H^0(k \widehat{C})\tensor H^0((a-k)\widehat{C} + E) \xrightarrow{} H^0(a\widehat{C} + E).
$$
\end{rem}

Combining \eqref{E205}, \eqref{E206}, \eqref{E207} and \eqref{E212}, we can compute $E_1(X)$.
Although the computation is not hard, we are not going to carry it out
as it is not very inspiring. All we need for Theorem \ref{THM000} is to show that $E_1(X)$ is a
rational function. For this purpose, we simply write
\begin{equation}\label{E213}
E_1(X) = \sum_i \sum_{(a,b)\in N_i\cap \BZ^2} P_i(a,b) t_1^a t_2^b
\end{equation}
where $N_i$ are a finite collection of
closed rational polyhedral cones in $\BR^2$, 
$N_i\cap \BZ^2$ are the lattice points contained in $N_i$ and
$P_i(a,b)$ are polynomials in $a$ and $b$. Here we allow $N_i$ to be degenerated, i.e., to
be contained in a linear subspace.
For example, the last term of \eqref{E205} can be written as
\begin{equation}\label{E214}
\sum_{2a > b > 0} = \sum_{2a\ge b\ge 0} - \sum_{2a = b \ge 0} - \sum_{2a \ge b = 0}
+ \sum_{a=b=0}
\end{equation}
Therefore, the rationality of $E_1(X)$ follows if we can show

\begin{prop}\label{PROP200}
For a closed rational polyhedral cone $N$ in $\BR^n$ and a polynomial
$P(x)\in \BZ[x_1, x_2,..., x_n]$, the series
\begin{equation}\label{E215}
\sum_{D\in N\cap \BZ^n} P(D) t^D \in \BZ[[M]]
\end{equation}
is rational, where $t^D = t_1^{d_1} t_2^{d_2} ... t_n^{d_n}$ for
$D = (d_1, d_2, ..., d_n)$ and $M$ is a submonoid of $\BZ^n$ containing
$N\cap \BZ^n$ and equipped with a monoid homomorphism $\deg: M\to \BN$
satisfying \eqref{E250}.
\end{prop}

The way we prove Proposition \ref{PROP200} also gives an algorithm to compute the series \eqref{E215},
which we will need later for the computation of Euler-Chow series of Del Pezzo surfaces.

First, for each $P(x)\in \BZ[x_1, x_2, ..., x_n]$, there exists a differential operator
\begin{equation}\label{E216}
Q = \sum_{i=1}^m f_i(t) \frac{\partial^{D_i}}{\partial t^{D_i}}
\end{equation}
such that
\begin{equation}\label{E217}
\sum_{D\in N\cap \BZ^n} P(D) t^D = Q\left(\sum_{D\in N\cap \BZ^n} t^D\right)
\end{equation}
where $f_i(t)\in \BZ[t_1, t_2, ..., t_n]$, $D_i\in \BN^n$ and
\begin{equation}\label{E218}
\frac{\partial^D}{\partial t^D} =
\frac{\partial^{d_1 + d_2 + ... + d_n}}{\partial t_1^{d_1} \partial t_2^{d_2}
... \partial t_n^{d_n}}
\end{equation}
for $D = (d_1, d_2, ..., d_n)$. Therefore, to show the rationality of \eqref{E215},
it suffices to show that of
\begin{equation}\label{E219}
\sum_{D\in N\cap \BZ^n} t^D.
\end{equation}

\begin{lem}\label{LEM202}
Suppose that $N\subset \BR^n$ is a closed rational simplicial cone of dimension $m\le n$,
i.e., it is generated by $m$ linearly independent rational vectors
${v}_1, {v}_2,
..., {v}_m \in \BQ^n$. Then \eqref{E219} is a rational function in $\BZ[[M]]$ with
$M$ a submonoid of $\BZ^n$ containing $N\cap \BZ^n$ and satisfying \eqref{E250}.
\end{lem}

\begin{proof}
After replacing ${v}_i$ by
$\lambda{v}_i\in \BZ^n$ for some $\lambda\in \BZ^+$,
we may assume ${v}_i\in \BZ^n$. Let
\begin{equation}\label{E231}
\Sigma_N = \BZ^n \cap \left\{ \sum_{i=1}^m a_i {v}_i: 0\le a_i < 1\right\}.
\end{equation}
Clearly, $\Sigma_N$ is a finite set and every ${v}\in N\cap \BZ^n$
can be uniquely written as
\begin{equation}\label{E232}
{v} = {w} +
\sum_{i=1}^m \lambda_i {v}_i
\end{equation} 
for some ${w}\in \Sigma_N$ and $\lambda_1, \lambda_2, ...,
\lambda_m\in \BN$. Thus
\begin{equation}\label{E230}
\sum_{D\in N\cap \BZ^n} t^D = \left(\sum_{{w}\in \Sigma_N}
t^{{w}}\right)
\prod_{i=1}^m \frac{1}{1 - t^{{v}_i}}
\end{equation}
is a rational function.
\end{proof}

To show that \eqref{E219} is rational for an arbitrary rational polyhedral cone $N$, it suffices
to subdivide $N$ into a finite union of simplicial cones which meet along faces \cite{S}. This proves
Proposition \ref{PROP200} and hence $E_1(X)$ is rational for a very general pair $(S, p)$.

\subsection{Some further comments}

If the pair $(S, p)$ fails to be very general, the corresponding $\Cox(X)$ might
still be finitely generated. It is interesting to study how $\Cox(X)$ and $E_1(X)$ vary
as $(S, p)$ does. To set this up, let us consider
\begin{equation}\label{E221}
\begin{split}
B &= \{ (S,p): \text{$S$ is a smooth quartic surface with $\Pic(S) = \BZ$},\\
&\quad \text{$p\in S$ is a point such
that the tangent plane of $S$ at $p$ cuts out}\\
&\quad\text{on $S$ a curve $C\in |L|$ with a single node}\}
\subset
|\CO_{\PP^3}(4)| \times \PP^3
\end{split}
\end{equation}
and the universal family $\CS = \{ (S, p, q): q\in S \} \subset B\times \PP^3$ over $B$.
Note that both $B$ and $\CS$ are complements of unions of countably many closed subvarieties in some projective varieties.

Clearly, $\CS/B$ has a section $P$ given by the map $B\to \CS$ sending $(S,p)$ to $(S,p,p)$. Let
$\X$ be the blow-up of $\CS$ along $P$. Obviously, at each point $b = (S,p)\in B$, the fiber
$\X_b$ of $\X/B$ at $b$ is exactly the blow-up $\Bl_p S$.

\begin{question}\label{Q002}
What is the set
$\Delta_M = \{ b\in B: \X_b \text{ is a MDS}\}$ in $B$?
Is it Zariski closed in $B$?
\end{question}

It is tempting to think that $\Delta_M$ consists of $(S,p)$ with the property
$K_{\widehat{C}} - q_1 - q_2\in \Pic(\widehat{C})_\tors$. This, however, is unlikely to
be true by a naive dimension count: the subvariety
\begin{equation}\label{E220}
\{(S,p)\in B: K_{\widehat{C}} - q_1 - q_2 \text{ is an $n$-torsion}\}
\end{equation}
has codimension $2$ in $B$ while the subvariety
\begin{equation}\label{E222}
\Delta_{M,n} = \{(S,p)\in B: h^0(a\widehat{C}) = 1 \text{ for } 0\le a < n,
h^0(n\widehat{C}) > 1\}
\end{equation}
has negative expected dimension for $n$ sufficiently large. 

Since $\widehat{C}$ is semi-ample if and only if $n\widehat{C}$ is movable for some
$n > 0$,
we see that $\Delta_M$ is the union of $\Delta_{M,n}$ and hence, a priori,
is a countable union of subvarieties of $B$.

Likewise, we want to know how $E_1(\X_b)$ varies:

\begin{question}\label{Q003}
Is $E_1(\X_b)$ a rational function for all $b\in B$?
\end{question}

\section{Transcendental Euler-Chow Series}
\label{SECTECS}

\subsection{Transcendence criteria}

We will obtain our first transcendence criterion based upon
the following algebraic result. 

To make the statement as general as possible, we work with $R[M]$ and $R[[M]]$ instead of
$\BZ[M]$ and $\BZ[[M]]$ for an arbitrary integral domain $R$. The rationality and algebraicity
of $f(t)\in R[[M]]$ are defined in an obvious way.

\begin{prop}\label{PROP006}
Let $M$ be a submonoid of $\BZ^m$ 
with a monoid homomorphism $\deg: M\to \BN$
satisfying \eqref{E250},
$J$ be a subset of $M$ and $R$ be an integral domain.
Suppose that there is a collection
$
\{\delta_\alpha\in \text{Hom}_\BZ(\BZ^m, \BR): \alpha\in A\}
$
of $\BZ$-linear functions $\delta_\alpha: \BZ^m\to \BR$ satisfying
\begin{itemize}
\item the minimum
\begin{equation}\label{E010}
\varepsilon_\alpha = \min_{D\in J} \delta_\alpha(D) \ge 0
\end{equation}
exists for every $\alpha\in A$;
\item $\{J_\alpha: \alpha\in A\}$ is an infinite set, where
\begin{equation}\label{E055}
J_\alpha = \{ D\in J: \delta_\alpha(D) = \varepsilon_\alpha\}
\end{equation}
for $\alpha\in A$.
\end{itemize}
Then
\begin{equation}\label{E061}
f(t) = \sum_{D\in J} a_D t^D\in R[[M]]
\end{equation}
is transcendental as long as $a_D \ne 0$ for every $D\in J$.
\end{prop}

\begin{proof}
Each $\delta\in \Hom_\BZ(\BZ^m, \BR)$ makes $R[M]$ into an $\BR$-graded ring by
\begin{equation}\label{E070}
R[M] = \bigoplus_{d\in \BR} \left(\bigoplus_{\delta(D) = d} R t^D \right).
\end{equation}
We can call $\delta(D)$ the weight of $t^D$ under this grading.

Let
\begin{equation}\label{E062}
f_\alpha(t) = \sum_{D\in J_\alpha} a_D t^D\in R[[M]].
\end{equation}
Since $\{ J_\alpha: \alpha\in A\}$ is an infinite set, the set
$\{f_\alpha(t): \alpha\in A\}\subset R[[M]]$ is also infinite since $a_D\ne 0$
for all $D\in J$.

Suppose that $f(t)$ is algebraic. Then there exists a nonzero polynomial
$F(t, x)\in R[M, x] = R[M][x]$
such that $F(t, f(t)) = 0$ in $R[[M]]$. We write
\begin{equation}\label{E056}
F(t, x) = \sum_{D\in M}\sum_{k=0}^\infty b_{D,k} t^D x^k
\end{equation}
where $b_{D,k}\in R$ vanishes outside finitely many pairs $(D, k)$.

Let $\Pi$ be the subset of $R[M,x]$ given by
\begin{equation}\label{E057}
\begin{split}
\Pi &= \bigg\{ G(t,x) = \sum_{D\in M}\sum_{k=0}^\infty c_{D,k} t^D x^k:\\
&\quad\quad
c_{D,k} = b_{D,k}\text{ or } c_{D,k} = 0 \text{ for each pair } (D,k)\bigg\}.
\end{split}
\end{equation}
Obviously, $\Pi$ is a finite set.

For each $\alpha\in A$, we let
\begin{equation}\label{E071}
\mu_\alpha = \min_{b_{D,k}\ne 0} (\delta_\alpha(D) + k \varepsilon_\alpha)
\end{equation}
and
\begin{equation}\label{E058}
G_\alpha(t, x) = \sum_{\delta_\alpha(D) + k \varepsilon_\alpha = \mu_\alpha} b_{D,k} t^D x^k.
\end{equation}
Obviously, $G_\alpha(t,x)\ne 0$, $G_\alpha(t,x)\in \Pi$
and we see that $G_\alpha(t, f_\alpha(t)) = 0$ 
by collecting the terms of $F(t, f(t))$ of the lowest weight $\mu_\alpha$ under the grading
given by $\delta_\alpha$.

Since $\Pi$ is finite and $\{ f_\alpha(t)\}$ is infinite, there exists $G(t,x) \in \Pi\backslash\{0\}$
such that $G(t, g(t)) = 0$ for infinitely many different $g(t)\in R[[M]]$. In other words,
the polynomial $G(t, x)$ has infinitely many roots in $R[[M]]$.
Obviously, this is impossible for an integral domain $R[[M]]$.
This proves that $f(t)$ is transcendental.
\end{proof}

\begin{cor}\label{Transcendence 1}
Let $X$ be a smooth projective variety of dimension $n$ with $\Pic(X)\isom \BZ^m$. If there
are infinitely many effective divisors $D\in\Pic(X)$ each generating an extremal ray of the effective cone $\NE^1(X)$ of $\Pic(X)$, then $E^1(X)$ is transcendental.  Moreover,  for every smooth projective variety $Y$
that dominates $X$ via a birational regular map $\pi: Y\to X$, $E^1(Y)$ is transcendental.
\end{cor}

\begin{proof}
Assume that $A$ is the set of all effective classes
in $M$ that generate extremal rays of $\NE^1(X)$. Since $\NE^1(X)$ is strongly convex, there exists
$\delta_\alpha \in \NM^{n-1}(X)$ for each $\alpha\in A$, such that $\delta_\alpha(D)\ge 0$ for
all $D\in \NE^1(X)$ and $\delta_\alpha(D) = 0$ if and only if $D$ lies on the ray $[\alpha]$ generated
by $\alpha$. The corresponding $\varepsilon_\alpha$ and 
$J_\alpha$ defined by \eqref{E010} and \eqref{E055} are exactly
$\varepsilon_\alpha = 0$ and $J_\alpha = [\alpha]\cap M$.
Here we are trying to apply Proposition \ref{PROP006} with $J = M$.

Obviously, $\{J_\alpha: \alpha\in A\}$ is an infinite set. Consequently,
\begin{equation}\label{E229}
\sum_{D\in M} a_D t^D
\end{equation}
is transcendental provided that $a_D \ne 0$ for all $D\in M$. It follows that
$E^1(X)$ is transcendental.

For $Y$ dominating $X$ via a birational regular map $\pi: Y\to X$, it is enough to apply the same
argument as above with $\pi^* \delta_\alpha$.
\end{proof}

Based upon Proposition \ref{PROP006}, we can deduce another criterion  with the following observation: for an arbitrary nonzero effective divisor $F$ of $X$,
\[
E^1(X) \text{ is trancendental} \iff (1 - t^F) E^1(X) \text{ is trancendental}.
\]
Note that
\begin{equation}\label{E059}
(1 - t^{F}) E^1(X)
= (1 - t^{F}) \sum h^0(D) t^D
= \sum (h^0(D) - h^0(D - F)) t^D,
\end{equation}
thus  the nonzero terms $t^D$ of $(1-t^F) E^1(X)$ satisfy that $h^0(D) > h^0(D-F)$.

For each integral divisor $F\ne 0$ on $X$,
let $L_F\subset M$ be the submonoid
\begin{equation}\label{E073}
L_F = \{ D\in \Pic(X): h^0(X, D) > h^0(X, D - F) \}.
\end{equation}
Or equivalently, $L_F$ consists of effective divisors $D$ such that
$F\not\subset D_f$.

In some special cases, as we will see, it is even true that  
\[
(1 - t^{F}) E^1(X)=E^1(F)
\]
under the pullback $\Pic(X) \to \Pic(F)$.

If the cone $\overline{\Conv(L_F)}\subset H^2(X,\BR)$
has infinitely many extremal rays generated by
classes in $L_F$, then we can apply Proposition \ref{PROP006} similarly to the proof of Corollary \ref{Transcendence 1} to conclude that
\begin{equation}\label{E074}
\sum_{D\in L_{F}} a_D t^D
\end{equation}
is transcendental provided that $a_D \ne 0$ for all $D\in L_{F}$.
It follows that $(1 - t^{F}) E^1(X)$ and hence $E^1(X)$ are transcendental.

\begin{cor}
\label{PROP203}
Let $X$ be a smooth projective variety of dimension $n$ with $\Pic(X)\isom \BZ^m$.
If there is an integral divisor $F$ on $X$ such that there are infinitely many $E\in L_F$ each
generating an extremal ray of the cone $\overline{\Conv(L_F)}$, then
$E^1(X)$ is transcendental. Moreover, for every smooth projective variety $Y$
that dominates $X$ via a birational regular map $Y\to X$, $E^1(Y)$ is transcendental.
\end{cor}

\begin{prop}
\label{ME5}
Let $X$ be the blow-up of $\mathbb{P}^{r_1}\times\cdots\times\mathbb{P}^{r_{p}}$ at a finite set $\Lambda$ of points, where $r_i\geq 2$ for all $i$ and $r_1>2$ if $p=1$. Assume that $\Lambda $ lies on a linear subspace $P$ of codimension $1$, i.e., the pull-back of some hyperplane $H_{i_0}\subset \mathbb{P}^{r_{i_0}}$. Let the closed immersion $i:\widehat{P}\hookrightarrow X$ be the proper transform of $P$,
where $\widehat{P}$ is the blow-up of $P$ at $\Lambda$ 
and $i$ induces a natural isomorphism $\text{Pic}(X)\overset{i^*}\simeq  \text{Pic}(\widehat{P})$ under our assumptions. Then
\begin{equation}\label{E351}
(1 - t^{\widehat{P}}) E^1(X)=E^1(\widehat{P}).
\end{equation}
In particular, the monoid
\[
L_{\widehat{P}} = \{ D\in \text{Pic}(X): h^0(X, D) > h^0(X, D - \widehat{P})\},
\]
is isomorphic to the monoid $M_{\widehat{P}}$ of effective line bundles on $\widehat{P}$.
Moreover, for every $D\in L_{\widehat{P}}$, 
\[
h^0(\widehat{P}, i^*\mathcal{O}(D))=h^0(X, D)-h^0(X, D - \widehat{P}).
\]
\end{prop} 

The proof of Proposition \ref{ME5} is based upon the following fact.

\begin{fact}\label{ME4}
Let $r_i\geq 1$ for $1\leq i\leq p$ and $X$ be the blow-up of  $\mathbb{P}^{r_1}\times\cdots\times\mathbb{P}^{r_{p}}$ at a finite set
$\Lambda=\{P_1,\cdots, P_n\}$ giving exceptional divisors $E_1, \cdots, E_n$. Assume that $H_i$ is the pull-back of the hyperplane class of $\mathbb{P}^{r_i}$. Then 
\begin{align}\label{EqME2}
h^0(X, \sum_{i=1}^{p}a_iH_i+\sum_{j=1}^nb_jE_j)=h^0(X, \sum_{i=1}^{p}a_iH_i +\sum_{j=1}^n \text{min}(b_j, 0)E_j ).
\end{align}
For $b_j\geq 0$, the space $H^0(X, \sum_{i=1}^{p}a_iH_i-\sum_{j=1}^nb_jE_j)$ can be identified with 
\begin{align}
\label{EqME3}
\{s\in \otimes_{i=1}^{p} H^0(\mathbb{P}^{r_i},  a_iH_i):\ \text{mult}(s, {P_j})\geq b_j
\text{ for all }  j \}.
\end{align} 
\end{fact}

\begin{proof}[Proof of Proposition \ref{ME5}]
We may assume that $P$ is the pull-back of the hyperplane of $\mathbb{P}^{r_1}$ defined by $x_{r_1}=0$. 

Assume that $D\in L_{\widehat{P}}$.
Let us choose $s\in H^0(X, D)\backslash H^0(X, D-\widehat{P})$. Clearly $ i^*s$ is a nonzero section of $H^0(\widehat{P}, i^*\mathcal{O}(D))$. Thus $i^*\mathcal{O}(D)\in M_{\widehat{P}}$.
Conversely, assume that $ i^*\mathcal{O}(D)\in M_{\widehat{P}}$. Since sections of $\mathcal{O}(D)$ and $ i^*\mathcal{O}(D)$ can be identified with polynomials of multiple degrees in (\ref{EqME3}), for each section $\widehat{s}\in H^0(\widehat{P}, i^*\mathcal{O}(D))$,
there exists $s\in H^0(X, D)$ such that $i^*s=\widehat{s}$. Consequently, $D\in L_{\widehat{P}}$.

To prove the equation, we assume that $D=\sum_{i=1}^{p}a_iH_i+\sum_{j=1}^nb_jE_j\in L_{\widehat{P}}$. 
There are three cases:
\begin{enumerate}
\item $a_i=0$ for all $i$;
\item $a_i>0$ for some $i$ and $b_j > 0$ for some $j$;
\item $a_i>0$ for some $i$ and $b_j\leq 0$ for all $j$.
\end{enumerate}
In Case (1), $h^0(D)>0$ and hence $b_j\geq 0$ for all $j$; the proposition holds trivially. Case (2) can be reduced to Case (3) by Equation (\ref{EqME2}). 

Thus it suffices to prove  the case that $D=\sum_{i=1}^{p}a_iH_i-\sum_{j=1}^nb_jE_j$
for all $a_i, b_j\geq0$. Identifying  $H^0(X,D)$ with the  space of multi-graded homogeneous polynomials in (\ref{EqME3}), then
\begin{align*}
&H^0(X, D-\widehat{P})\simeq \{s\in H^0(X,D):\ x_{r_1} \text{ is a factor of }s \};\\
&\frac{H^0(X, D)}{H^0(X, D-\widehat{P})}\simeq \{s\in   \otimes_{i=1}^{p} H^0(\mathbb{P}^{r_i-\delta_{i1}},  a_iH_i):\ \text{mult}(s, {P_j})\geq b_j \text{ for all } j \},
\end{align*}
where  $\delta_{11}=1$ and $\delta_{i1}=0$ if $i\neq 1$. Obviously the latter is isomorphic to $H^0(\widehat{P}, i^*\mathcal{O}(D))$ by applying (\ref{EqME3}) again.
\end{proof}

\subsection{Mukai's construction and generalizations}

In \cite{Mu2}, Mukai has constructed a family of smooth projective varieties with infinitely generated Cox ring; by establishing an isomorphism between the Cox ring of these varieties and the invariant ring of an action of Nagata type,  he thus obtained a family of counterexamples to Hilbert's 14th problem. 

Mukai's theorem \cite[Theorem 3]{Mu2} can be reformulated as follows: 

\begin{thmm}[Mukai] 
Let $r>2$ and $X$ be the blow-up of $(\mathbb{P}^{r-1})^{p-1}$ at $q>r$ points in very general position. 
Assume that 
\begin{align}\label{EqME1}
\frac{1}{p}+\frac{1}{r}+\frac{1}{q-r}\leq 1.
\end{align}
Then $\text{Cox}(X)$ is  infinitely generated. When $p=2$, it is the result in  \cite{Mu1}. 
\end{thmm}

By ``$n$ points $\{P_1, \cdots, P_n \} \in\mathbb{P}^{r-1}$ in very general position", we mean that any $r$ points of $\{P_1, \cdots, P_n \}$ after any finite sequence of Cremona transformations span $\mathbb{P}^{r-1}$. Here a Cremona transformation is a birational map of the form $\sigma_1 \circ \Psi\circ\sigma_2$, where $\sigma_1, \sigma_2\in\text{Aut}(\mathbb{P}^{r-1})$, and 
 \[\xymatrix{
\Psi: \mathbb{P}^{r-1}\ar@{.>}[r] &\mathbb{P}^{r-1}, & & (x_1,\cdots,x_r)\mapsto (\frac{1}{x_1}, \cdots, \frac{1}{x_r}).
}\]
 By ``$n$ points $\{P_1, \cdots, P_n \} \in (\mathbb{P}^{r-1})^{p-1}$ in very general position", we mean that for $1\leq i\leq p-1$, the $i^{th}$ components $\{P_1^{(i)}, \cdots, P_n^{(i)} \}\in\mathbb{P}^{r-1}$ are in very general position.

Note that a key fact in the proof of \cite[Theorem 3]{Mu2}  is \cite[Lemma 3]{Mu1}, which says that exceptional divisors are indispensable as generators of the Cox ring. We observe that the proof of \cite[Lemma 3]{Mu1} implies that
exceptional divisors generate extremal rays of the effective cone of $X$. Thus the proof of \cite[Theorem 3]{Mu2} implies that the effective cone of $X$ has  infinitely many extremal rays. Therefore, $E^1(X)$ is transcendental by Corollary \ref{Transcendence 1}. 

\begin{cor}\label{ME1}
Let $r>2$ and $X$ be the blow-up of $(\mathbb{P}^{r-1})^{p-1}$ at $q$ points in very general position, where $p, q, r$ satisfy \eqref{EqME1}
and $q>r$.
Then the effective cone of $X$ has  infinitely many extremal rays and  the Euler-Chow series $E^1(X)$ is  transcendental. 
\end{cor}
 
With Corollary \ref{Transcendence 1} and Proposition \ref{ME5}, Corollary \ref{ME1} can be generalized in two directions as follows. 
\begin{thmm}[Theorem \ref{THM001}] 
For every  pair of integers $p>1$ and $r>2 $, let $q_0(r,p)$ be the minimal positive integer $q$ greater than $r$ and satisfying \eqref{EqME1}.
Then $E^1(X)$ is transcendental in the following cases:
\begin{enumerate}
\item $X$ is the blow-up of $(\mathbb{P}^{r-1})^{p-1}$ at $\Lambda$, where $r>2$, $p>1$, $\Lambda$  is a finite set of points in $(\mathbb{P}^{r-1})^{p-1}$ and contains $q_0(r,p)$ points  in very general position.

\item $X$ is the blow-up of the product $\mathbb{P}^{r_1-1}\times\cdots\times\mathbb{P}^{r_{p-1}-1}$ at a finite set $\Lambda$, where $p>1$, $\Lambda$ lies on a  linear subspace $(\mathbb{P}^{r_0-1})^{p-1}$ with $2<r_0\leq \text{min}_{i=1}^{p-1}(r_i)$ and contains  $ q_0(r_0,p) $   points in very general position  as points of $(\mathbb{P}^{r_0-1})^{p-1}$. 
\end{enumerate}
\end{thmm}
 
\begin{proof} 
Case (1) is a direct consequence of Corollaries \ref{ME1} and \ref{Transcendence 1}. 
Case (2) follows from Case (1) and  Proposition \ref{ME5}. Note that Proposition \ref{ME5} can be applied inductively on dimension to the proper transform of every linear subspace containing $(\mathbb{P}^{r_0-1})^{p-1}$. 
\end{proof}

\subsection{Elliptic fibration}

The purpose of this subsection is to prove that $E^1(X)$ is transcendental for some elliptic fibration.
\begin{thmm}[Theorem \ref{THM002}]  
 $E^1(X)$ is transcendental in the following cases:
\begin{enumerate}
\item $X$ is the blow-up of $\mathbb{P}^2$ at a finite set $\Lambda$,
where $\Lambda$  contains  the intersection of two very general cubic curves.

\item $X$ is the blow-up of $\mathbb{P}^3$ at a finite set $\Lambda$,
where $\Lambda$  contains  the intersection of three very general quadrics.

\item $X$ is the blow-up of  $\mathbb{P}^r$ at a finite set $\Lambda$,
where $\Lambda$ lies on a linear subspace $\mathbb{P}^2\subset \mathbb{P}^r$
and contains the intersection of two very general cubics.

\item $X$ is the blow-up of  $\mathbb{P}^r$ at a finite set $\Lambda$,
where $\Lambda$ lies on a linear subspace $\mathbb{P}^3\subset \mathbb{P}^r$
and contains the intersection of three very general quadrics.
\end{enumerate}
\end{thmm}

\begin{proof}
With Proposition \ref{ME5}., Case (3) and  (4) follows from Case (1) and (2) respectively.
The proof of Case (1) and (2)  makes use of the facts that $X$ is an elliptic fibration over $\mathbb{P}^1$ or $\mathbb{P}^2$. 

Case (2) is a consequence of Corollary \ref{PROP203} by setting $F = \widehat{Q}$ and Proposition \ref{PROP202},
the latter showing that there are infinitely many $(-1)$-curves on
the proper transform $\widehat{Q}$ of
a general member of the net of quadrics. Case (1) follows from a similar proof as Proposition \ref{PROP202}.
\end{proof}

\begin{prop}\label{PROP202}
Let $X$ be the blow up of $\PP^3$ at the base locus $\Lambda$ of a very general net of quadrics
in $\PP^3$ and let $\widehat{Q}$ be the proper transform of a general member of the net. Then the image of 
\begin{equation}\label{E239}
L_{\widehat{Q}} = \{ D\in \Pic(X): h^0(X, D) > h^0(X, D - \widehat{Q})\},
\end{equation}
 under the injection
$\Pic(X)\hookrightarrow \Pic(\widehat{Q})$ as a submonoid of $ \Pic(\widehat{Q})$
contains infinitely many $(-1)$-curves on $\widehat{Q}$.
\end{prop}

\begin{proof}
Let $H, E_1, E_2, ..., E_8$ be the generators of $\Pic(X)$, where $H$ is the pullback of the
hyperplane divisor and $E_1, E_2, ..., E_8$ are the exceptional divisors of the blow-up $X\to \PP^3$.

The net of quadrics gives a rational map $\PP^3\dashrightarrow \PP^2$ with $\Lambda$
the indeterminacy locus. Blowing up $\Lambda$ gives a regular map $f: X\to \PP^2$, which is
a fibration of elliptic curves with sections $E_1, E_2, ..., E_8$. Each fiber of $f$ is the
proper transform of the
intersection of two quadrics of the net and the pull back $f^{-1}(\Gamma)$ of a line
$\Gamma\subset \PP^2$ is the proper transform of a quadric of the net. So 
$\widehat{Q} = f^{-1}(\Gamma)$ for a general line $\Gamma\subset\PP^2$.

Let $X_\eta$ be the generic fiber of
$f: X\to \PP^2$, $J(X_\eta) = \Pic_0(X_\eta)$ be the Jacobian of $X_\eta$ and $A\subset J(X_\eta)$
be the intersection of $J(X_\eta)$ with the subgroup generated by $H, E_1, E_2, ..., E_8$. For each $a \in A$, we have
an automorphism $\phi_a: X_\eta\to X_\eta$ by taking $p$ to $p + a$; $\phi_a$ corresponds to a
birational self map $\phi_a: X\dashrightarrow X$. More explicitly, for each
$a = dH + m_1E_1 + m_2E_2 +
... + m_8 E_8$ satisfying $4d + m_1 + m_2 + ... + m_8 = 0$,
$\phi_a: X\dashrightarrow X$ is a birational map sending $p\in X_b$ to $p+a\in X_b$
on a general fiber $X_b$ of $f$. Obviously, $\phi_a$ preserves the fiberation $X/\PP^2$, i.e.,
$f\circ \phi_a = f$.

This map can be extended to all irreducible fibers of $f$ since
the map $p\to p+a$ is well defined on a rational curve with one node whose Picard
group is $G_m$ and a rational curve with one cusp whose Picard group is $G_a$.
And since $f$ has only finitely many
reducible fibers, $\phi_a$ is an isomorphism 
\begin{equation}\label{E240}
\phi_a: X\backslash f^{-1}(\Delta)\xrightarrow{\sim} X\backslash f^{-1}(\Delta)
\end{equation}
in codimension one, where $\Delta\subset\PP^2$ is the finite set of points $b$
with reducible fiber $X_b$. In particular, $\phi_a$ induces an isomorphism
$\phi_a: \widehat{Q}\xrightarrow{\sim} \widehat{Q}$.

For each $a\in A$, $G_a = \phi_a(E_1)$ is a rational section of $f$. And since $\phi_a$ is
an isomorphism on $\widehat{Q}$, $G_a\cdot \widehat{Q}$ is a $(-1)$-curve on $\widehat{Q}$.
Clearly, $G_a$ meets $\widehat{Q}$ properly and hence $G_a\in L_{\widehat{Q}}$. And the orbit
$\{G_a: a\in A\}\subset \Pic(X)$ of $E_1$ under the action of $A$ is obviously infinite. We are done.
\end{proof}

\section{Euler-Chow Series of Del Pezzo surfaces}
\label{SECDPS}

\subsection{Some basic facts on $\Bl_\Lambda \PP^2$}

In the computation of Euler-Chow series of Del Pezzo surfaces, we will need some statements on divisors of the surfaces to be discussed as follows. In this subsection, we assume that $X$ is the blow-up of $\PP^2$ at $\Lambda$
with $\Lambda$ being either $r\le 9$ points in general position (very general
position if $r=9$) or
the intersection of two very general cubics.  
Note that  $-K_X$ is nef and effective. If $r\leq 8$, then $X$ is a Del Pezzo surface and  $-K_X$ is ample. 

\begin{lem}\label{LEM004}
For a non-zero effective divisor $D$, $-K_X D > 0$ unless $r = 9$ and $D = -d K_X$ for some $d> 0$.
For an integral curve $D\subset X$, $D^2 \ge -1$ and
\begin{itemize}
\item $D$ is nef if $D^2 \ge 0$;
\item $D$ is a $(-1)$-curve if $D^2 = -1$.
\end{itemize}
\end{lem}

\begin{proof}
When $r\le 8$, $-K_X$ is ample, thus  $-K_X D > 0$. 

When $r=9$, assume that $D = dH - \sum_{i=1}^9 m_i E_i$ and $C$ is a general member of $|-K_X|$. Note that $C$ is a smooth elliptic
curve under our assumptions on $\Lambda$. Suppose that $K_X D = 0$
and $D$ is not a multiple of $C$. Replacing $D$ by $D - \lambda C$, 
we may assume that $D$ meets $C$ properly.
Restricting $D$ to $C$, let $P_i = C\cap E_i$, then
\begin{equation}\label{E234}
\CO_X(D)|_C = d H - \sum_{i=1}^9 m_i P_i = 0 \text{ in }\Pic(C).
\end{equation}
When $\Lambda$ is a set of $9$ points in very general position,
$P_1, P_2, ..., P_9$ are $9$ very general points on $C$; therefore, there are no relations between them
and $H$ in $\Pic(C)$ and \eqref{E234} cannot happen. When $\Lambda$ is the intersection of two very general
cubics, the only relation between $P_1, P_2, ..., P_9$ and $H$ is
\begin{equation}\label{E233}
3H - \sum_{i=1}^9 P_i = 0
\end{equation}
in $\Pic(C)$. That is, $D = -(d/3) K_X$. Contradiction.

Let $D$ be an integral curve.
Since $(K_X+D)D = 2p_a(D) -2 \ge -2$ and $K_X D \le 0$, we conclude that $D^2 \ge -2$, where
$p_a(D)$ is the arithmetic genus of $D$. And $D^2 = -2$ only if $K_X D = 0$, which can only happen
when $r = 9$ and $D = -d K_X$; but then we have $D^2 = 0$.
 If $D^2 \ge 0$, obviously $D$ is nef; if $D^2 = -1$, then $p_a(D) = 0$ and $D$ is a $(-1)$-curve.
\end{proof}

\begin{lem}\label{LEM003}
Let $D$ be a nonzero nef divisor on $X$. Then
\begin{itemize}
\item
$D$ is effective. 
\item
$H^1(D) = 0$ and
\begin{equation}\label{E001}
h^0(D) = \frac{(D - K_X) D}{2} + 1,
\end{equation}
unless $\Lambda$ is the intersection of two general cubic curves and $D = -dK_X$ for some $d >0$.
\item
$h^0(D) > 1$ unless $\Lambda$ is a set of $9$ points in general
position and $D = - dK_X$ for some $d > 0$.
\end{itemize}
\end{lem}

\begin{proof}
By Riemann-Roch,
\begin{equation}\label{E075}
h^0(D) - h^1(D) + h^2(D) = \frac{(D - K_X) D}{2} + 1.
\end{equation}
By Serre duality, $h^2(D) = h^0(K_X - D)$. Since $D$ is nef and $-K_X$ is nef, $\frac{(D - K_X) D}{2}\ge 0$  and $H(K_X -D) < 0$ for every ample divisor $H$, which implies that $K_X -D$ is not effective and
thus $h^2(D) = h^0(K_X -D) = 0$.    Therefore,
\begin{equation}\label{E228}
h^0(D) = \frac{(D - K_X) D}{2} + 1+h^1(D)\ge 1, 
\end{equation}
and $D$ is
effective.  

By Lemma \ref{LEM004}, $-K_X D > 0$ unless $r = 9$ and $D = -d K_X$.
If  $-K_X D > 0$, clearly $h^0(D) > 1$. 
Then $(D-K_X)^2 > 0$,  thus
$D - K_X$ is big and nef. We have $H^1(D)=0$ by Kawamata-Viehweg vanishing theorem
and \eqref{E001} follows.  

If $-K_X D = 0$, $r = 9$ and $D = -dK_X$. If $\Lambda$ is a set of $9$ general points,
$h^0(D) = 1$ and we still have $h^1(D) = 0$ and \eqref{E001}.
If $\Lambda$ is the intersection of two general cubics and
$D = -dK_X$, $h^0(D) > 1$.
\end{proof}

\begin{lem}\label{LEM005}
Every effective divisor $D$ on  $X$ can be uniquely written as
\begin{equation}\label{E002}
D = A + m_1 I_1 + m_2 I_2 + ... + m_a I_a
\end{equation}
in $\Pic(X)$ for some nef and effective divisor $A$ and some set of disjoint $(-1)$-curves $I_1, I_2, ..., I_a$ such that $A I_k = 0$ for all $k$ and $m_1,m_2,...,m_a \in \BZ^+$. In addition, $h^0(D) = h^0(A)$.
\end{lem}

\begin{proof}
Indeed, \eqref{E002} is the Zariski decomposition of $D$ and
$A$ is the maximal element of the set of nef divisors $B$ such that
$D - B\in \NE^1(X)$.

Let $D_f$ be the fixed part of  $|D|$ and  write  $D =  D_\mu+D_f $ as in \eqref{E076}. 
Note that  $D_\mu$ is nef as it is easy to verify that $C\cdot D_\mu\ge 0$ for every integral curve $C$. We let
\begin{equation}\label{E227}
D = D_\mu + D_f = A + F
\end{equation}
where $A\supset D_\mu$, $F\subset D_f$ and $A$ is nef and maximal in the sense that $A + F'$ is not nef for every nonzero effective divisor $F'\subset F$.

First, every  irreducible component $I$ of $F$ is a $(-1)$-curve. Otherwise, by Lemma \ref{LEM004} $I$ is nef, so would be $A+I$, contradiction with the choice of $A$. 

Second, $I_1\cdot I_2=0$ for  two distinct irreducible components  $I_1$ and $I_2$ of $F$.
Otherwise, if $I_1\cdot I_2 > 0$,  then $I_1+I_2$ is nef and $A+I_1+I_2$ is nef. This contradicts with our choice of $A$.

Last, $A\cdot I = 0$ for  every  irreducible component $I$ of $F$. Otherwise, if $A\cdot I>0$, then $A+I$ is nef, this contradicts with our choice of $A$.

In conclusion,
\begin{equation}\label{E226}
D = A + F = A + m_1 I_1 + m_2 I_2 + ... + m_a I_a
\end{equation}
with required properties. Clearly, this representation of $D$ is unique since
$m_k = -D I_k$ for all $k$. Finally, since $\sum m_k I_k\subset D_f$,
$h^0(D) = h^0(A)$.
\end{proof}

\begin{rem}\label{REM300}
Actually we have proved the following statement: for every effective divisor $D$, let $S$ be the set of $(-1)$-curves $I$ satisfying $D\cdot I<0$, then $A=D-\sum_{I\in S}( I\cdot D)I$ is nef and effective; $A\cdot I=0$ for every $I\in S$; $I_1\cdot I_2=0$ for $I_1, I_2\in S$. Moreover, the unique decomposition implies a bijection from $|D|$ to $|A|$.
\end{rem}

\begin{lem}\label{LEM007}
Let $D$ be an effective divisor on  $X$.
Then $H^1(D) = 0$ and \eqref{E001} holds if and only if $DI\ge -1$ for all
$(-1)$-curves $I\subset X$, unless $\Lambda$ is the intersection of two general cubic curves and $A = -dK_X$ for some $d >0$ in \eqref{E002}.
\end{lem}

\begin{proof}
We write $D=A +m_1 I_1 +m_2 I_2 + ... + m_aI_a$ in the form of \eqref{E002}. Clearly, $DI\ge -1$ for all $(-1)$-curves $I$ if and only if $m_i =1$ for all $i$. For every $I_j$, consider  the exact sequence
\begin{equation}\label{E079}
 H^1(\CO_X(D)) \xrightarrow{} H^1(\CO_{I_j}(-m_j)) \xrightarrow{} H^2(\CO_X(D - I_j)),
\end{equation}
where the last term vanishes because $D - I_j$ is effective. If $H^1(D) = 0$, then $H^1(\CO_{I_j}(-m_j))=0$; as $m_j\in\BZ^+$, thus $m_j=1$.

Conversely, suppose that $m_i = 1$ for all $i$.
We observe that if $B I = 0$ and $H^1(B) = 0$,
then $H^1(B+I) = 0$ by the exact sequence
\begin{equation}\label{E078}
H^1(\CO_X(B)) \xrightarrow{} H^1(\CO_X(B+ I)) \xrightarrow{} H^1(\CO_I(-1)).
\end{equation}
So we can inductively show that
$H^1(D) = H^1(A + I_1 + I_2 + ... + I_a) = 0$.
\end{proof}

\begin{lem}\label{LEM006}
Let $D$ be an effective divisor on $X$. Suppose that $r\ge 2$. Then
\begin{itemize}
\item $D$ is nef if and only if $DI \ge 0$ for all $(-1)$-curves $I\subset X$.
\item $D$ is ample if and only if $DI > 0$ for all $(-1)$-curves $I\subset X$
and $D$ is not a multiple of $-K_X$ when $r = 9$.
\item $D$ is ample if and only if
\begin{equation}\label{E077}
D = m(-K_X) + F
\end{equation}
for some $m\in \BZ^+$ and some nef divisor
$F$ that is not ample and not a multiple of $-K_X$ when $r =9$.
\end{itemize}
\end{lem}

\begin{proof}
The first statement follows directly from Lemma \ref{LEM005}.

Let $m$ be the minimum of $DI$ for all $(-1)$-curves $I\subset X$ and let
\[
F = D + m K_X.
\]
Clearly, $FI \ge 0$ for all $(-1)$-curves $I$ and 
$FI = 0$ for some $(-1)$-curve $I$;
hence $F$ is nef and not ample. If $D$ is ample, then $m > 0$ and $F$ is not a multiple
of $-K_X$ when $r = 9$, obviously.

On the other hand,
suppose that $m > 0$ and $F$ is nef and not a multiple of $-K_X$ when $r=9$.
Let $C$ be an integral curve. Then $-K_X C > 0$ unless
$C = -K_X$ and $r = 9$. When $C = -K_X$ and $r=9$, we have $-K_X C = 0$ and $FC > 0$. In
conclusion, $DC > 0$ for all integral curves $C$. So $D$ is ample.
This proves both the second and third statements.
\end{proof}
 
\subsection{Euler-Chow series}

Let $X = P_r$ be the blow-up of $\PP^2$ at $r\le 8$  points in general position. Write $E_1(X) = f_r(t_0, t_1, ..., t_r)$, 
where $t_0 = t^H$ and $t_i = t^{E_i}$ for $i=1,2,...,r$. Our aim is to develop a recursive formula
for $\ECS{1}$ in three steps.
First, rearrange $E_1(X)$  in terms of  $N_X(t)$ given in \eqref{E006}; second, express $N_X(t)$ in terms of the simplest series $\sum_{A\text{ nef}} t^A$ defined over a nef cone; last, write $\sum_{A\text{ nef}} t^A$ as a rational function by decomposing the nef cone of $P^r$ as a union of simplicial rational polyhedral cones. In particular, detailed computation is given for $r\leq 4$.

For each set $S = \{I_1, I_2, ..., I_a\}$ of disjoint $(-1)$-curves, let $M_S$ be the semigroup $\{ \sum_{i=1}^a m_i I_i | m_i\in\mathbb{Z}^+ \}$. 
By Lemma \ref{LEM005}, every effective divisor $D$ on $X$ is of the form \eqref{E002} and
$h^0(D) = h^0(A)$, thus  
\begin{align}\label{E005}
E_1(X) = \sum_{S}\sum_{I\in M_S} 
\sum_{\genfrac{}{}{0pt}{}{A\in S^{\perp}}{A\text{ nef}}} h^0(A) t^{A+I} 
=\sum_{S}\left(  \sum_{I\in M_S}t^I  \cdot  \sum_{\genfrac{}{}{0pt}{}{A\in S^{\perp}}{A\text{ nef}}} h^0(A) t^A     \right)  
\end{align}
where $S$ runs over all sets of disjoint $(-1)$-curves, $S^\perp$ is the group of divisors $B$ satisfying $BI = 0$ for all $I\in S$,  and $ \sum_{I\in M_S}t^I=\tfrac{t^{I_1+I_2+...+I_a}}{(1-t^{I_1})(1-t^{I_2}) ... (1 - t^{I_a})}$.

So naturally we turn to consider the series
\begin{equation}\label{E006}
N_X(t) = \sum_{A\text{ nef}} h^0(A) t^A = \sum_{A\text{ nef}}
\left(\frac{(A - K_X)A}{2}+1\right) t^A.
\end{equation}
Let $g_r(t_0,t_1,...,t_r) = N_X(t)$ for $X = P_r$. We  first express $f_r$ in terms of $g_r$.

Note that for each $S$ in \eqref{E005}, there is a map $\pi_S: X\to X_S$ of Del Pezzo surfaces given by contracting the $(-1)$-curves in $S$. Obviously,  $S^\perp = \pi_S^*(\Pic(X_S))$.
Thus, we have
\begin{equation}\label{E080}
\sum_{\genfrac{}{}{0pt}{}{A\in S^{\perp}}{A\text{ nef}}} h^0(A) t^A = \pi_S^*(N_{X_S}(t)),
\end{equation}
where $ \pi_S^*(N_{X_S}(t))=\sum_{A\text{ nef}, A\in \Pic(X_S)} h^0(A) t^{\pi_S^*(A)}$.

Let $a=\#S$. Clearly, if $a\neq r-1$, $X_S\isom P_{r-a}$; if $a = r-1$,
$X_S$ is either $P_1$ or $\BF_0= \PP^1\times\PP^1$.
In particular, if $S=S_k := \{ E_{k+1}, E_{k+2}, ..., E_r \}$,  then $X_{S}=P_k$ and the sum in \eqref{E080} is $ g_{k}(t_0, t_1,..., t_{k})$; if  $S= T := \{ H - E_1 - E_2, E_3, ..., E_r\}$,  then $X_S=\BF_0$ and the sum in \eqref{E080} is $ q(t_0t_1^{-1}, t_0t_2^{-1})$, where $q(t_1, t_2) = \tfrac{1}{(1-t_1)^2(1-t_2)^2}$ is the Euler-Chow series of $\BF_0$ with  $t_1 = t^{H_1}$ and  $t_2 = t^{H_2}$ for two rulings  $\{ H_1, H_2\}$   of $\BF_0$.

To locate all sets of disjoint $(-1)$-curves on $P_r$, we consider the group
$\Phi\subset \Aut(\Pic(P_r))$ generated by $\Sigma_r$ and $\varphi_{abc}$ for all distinct integers
$1\le a, b, c\le r$, where the action of the symmetric group $\Sigma_r$ of $\{1,2,...,r\}$ on $\Pic(X)$ is defined by
sending $H\to H$ and $E_l\to E_{\sigma(l)}$ for $\sigma\in \Sigma_r$, and  $\varphi_{abc}$ is   given by
\begin{equation}\label{E013}
\begin{split}
\varphi_{abc}(H) &= 2H - E_a - E_b - E_c\\
\varphi_{abc}(E_a) &= H - E_b - E_c\\
\varphi_{abc}(E_b) &= H - E_c - E_a\\
\varphi_{abc}(E_c) &= H - E_a - E_b\text{ and}\\
\varphi_{abc}(E_i) &= E_i\text{ for } i\ne a,b,c.
\end{split}
\end{equation}

Indeed, for Del Pezzo surfaces $P_r$ with $3\leq r \leq 8$,
$\Phi$ are respectively Weyl groups of type $A_1\times A_2$,$A_4$, $D_5$, $E_6$, $E_7$ and $E_8$, see \cite[Theorem 23.9]{Man}. 

\begin{lem}\label{LEM010}
Let $\Pi$ be the set of $(-1)$-curves on $X$. Then
\begin{itemize}
\item
$\Phi(\Pi) = \Pi$, i.e., every element of $\Phi$ induces a
permutation of $\Pi$. Indeed, $\Phi $ is isomorphic to  the group $ \Aut(\Pi)$ of bijections $\varphi:\Pi\to \Pi$ preserving the
intersection pairing, i.e., $\varphi(I_1) \varphi(I_2) = I_1 I_2$ for all $I_1, I_2\in\Pi$.
\item
For every subset $S$ of disjoint $(-1)$-curves,
there exists $\varphi\in \Phi$ such that $\varphi(S) = S_{k}$ or $T$, where  $k=r-\#S$. 
\end{itemize}
\end{lem}

The proof of the above lemma follows the argument for \cite[V, Ex 4.15]{H}, which we will omit due to its tedious nature. In \cite[V, 4.10.1]{H}, $\Phi$ is called the group of {\em automorphisms of the configuration
of lines} on $X$, as $\Pi$ consists of lines on $X$ under the map $X\to \PP^N$ by $|-K_X|$.

By Lemma \ref{LEM010}, every set  of disjoint $(-1)$-curves lies in the orbit of
$S_k$ or $T$ under $\Phi$. With this in mind, we can rewrite \eqref{E005} as
\begin{equation}\label{E083}
\begin{split}
&\quad f_r(t_0,t_1,...,t_r)\\
&= \sum_{k=0}^r \frac{1}{|\Phi_{S_k}|}
\sum_{\varphi\in \Phi} \varphi\left(g_k(t_0,t_1,...,t_k)\prod_{j=k+1}^r \frac{t_j}{1-t_j}\right)\\
&+ \sum_{\varphi\in\Phi} \frac{1}{|\Phi_T|}
\varphi\left(q(t_0t_1^{-1},t_0t_2^{-1})\frac{t_0t_1^{-1}t_2^{-1}}{1 - t_0t_1^{-1}t_2^{-1}}\prod_{j=3}^r \frac{t_j}{1-t_j}\right)\\
\end{split}
\end{equation}
where $\Phi_{S_k}$ and $\Phi_T$ are subgroups of  $\Phi$ consisting of $\varphi$ with $\varphi(S_k)=S_k$ and $\varphi(T)=T$ respectively. 

\subsection{Computation of $N_X(t)$}
To compute $N_X(t)$,
we can apply the algorithm in the proof of Proposition \ref{PROP200}.

Let
\begin{equation}\label{E084}
\rho_r(t_0, t_1, ..., t_r) = L_X(t) = \sum_{A\text{ nef}} t^A.
\end{equation}
As in  \eqref{E217}, because of \eqref{E001},
there exists a second order differential operator $Q$ such that $N_X(t) = Q(L_X(t))$, where $Q$ is defined as follows:
\begin{equation}\label{E085}
Q = \left(\frac{t_0^2}{2} \frac{\partial^2}{\partial t_0^2} + 2 t_0 \frac{\partial}{\partial t_0}\right)
- \frac{1}{2}\sum_{k=1}^r t_k^2 \frac{\partial^2}{\partial t_k^2} +1.
\end{equation}

As a result, the computation of $N_X(t)$ comes down to  that of $L_X(t)$,
 the formal sum of
$t^A$ over all the lattice points $A$ in the nef cone $\NM^1(X) \subset H^2(X,\BR)$ of $X$. 
As $\NM^1(X)$ is a rational polyhedral cone given by Lemma \ref{LEM006}, thus we can follow the
algorithm in the proof of Proposition \ref{PROP200} and compute $L_X(t)$ by subdividing $\NM^1(X)$ into
simplicial cones.

Now we compute $L_X(t)$ for $X=P_r$ when $r\leq 4$.

Case $r=1$. $\text{NM}^1(P_1)=\{a_0H+a_1(H-E_1) | a_0, a_1\in \mathbb{Z}_{\geq 0}\}$. Therefore, \begin{equation}
\rho_1(t_0, t_1)=\tfrac{1}{(1-t_0)(1-t_0/t_{1})}.
\end{equation}

Case $r=2$. $\text{NM}^1(P_2)=\{a_0H+a_1(H-E_1)+a_2(H-E_2) | a_0, a_1, a_2\in \mathbb{Z}_{\geq 0}\}$. Therefore, 
\begin{equation}
 \rho_2(t_0, t_1, t_2)=\tfrac{1}{(1-t_0)(1-t_0/t_{1})(1-t_0/t_2)}.
\end{equation}

Case $r=3$. $\text{NM}^1(P_3)=\{a_0H-\sum_{i=1}^3a_iE_i | a_i\in \mathbb{Z}_{\geq 0}, 0\leq i\leq 3; a_0\geq a_i+a_j,  1\leq i\neq j\leq 3\}$, thus $\text{NM}^1(P_3)=\{a_0H+\sum_{i=1}^3a_i(H-E_i) | a_i\in \mathbb{Z}_{\geq 0}, 0\leq i\leq 3\}\cup\{a_0(2H-\sum_{i=1}^3E_i)+\sum_{i=1}^3a_i(H-E_i) | a_i\in \mathbb{Z}_{\geq 0}, 0\leq i\leq 3\}$. Therefore,
\begin{align}
& \rho_3(t_0, t_1, t_2, t_3)=\tfrac{1}{(1-t_0^2/(t_1t_2t_3))(1-t_0/t_{1})(1-t_0/t_2)(1-t_0/t_3)}\\
&+\tfrac{1}{(1-t_0)(1-t_0/t_{1})(1-t_0/t_2)(1-t_0/t_3)}-\tfrac{1}{(1-t_0/t_{1})(1-t_0/t_2)(1-t_0/t_3)}. \nonumber
\end{align}

Case $r=4$. We have
\[
\text{NM}^1(P_4)=\{a_0H-\sum_{i=1}^4a_iE_i | a_i\in \mathbb{Z}_{\geq 0}, 0\leq i\leq 4; a_0\geq a_i+a_j,  1\leq i\neq j\leq 4\}.
\]
In the equations below, notations like $\text{NM}^1(P_4)\cap\{\cdots\}$ stand for subsets of  $\text{NM}^1(P_4)$ defined by inequalities inside the braces; all the $b_i$'s are assumed to go through $\mathbb{Z}_{\geq 0}$ unless otherwise noted.

Note that a decomposition of $\text{NM}^1(P_4)$ as a union of simplicial rational polyhedral cones can be given as follows: \begin{equation}
\text{NM}^1(P_4)=\mathcal{C}_1\cup\mathcal{C}_2\cup \bigcup_{i=1}^4
(\mathcal{C}_{3,i}\cup\mathcal{C}_{4,i}),
\end{equation}
where
\begin{align*}
&\mathcal{C}_1=\text{NM}^1(P_4)\cap\{\sum_{j=1}^4a_j\leq a_0 \}=\{ b_0H+\sum_{j=1}^4 b_j(H-E_j) \},\\
& \mathcal{C}_2=\text{NM}^1(P_4)\cap\{\sum_{j=1}^4a_j>a_0;  \sum_{j=1}^4a_j \leq a_i+a_0,  \forall   i >0\}\\
&=\{ b_0(-K)+\sum_{j=1}^4 b_j(H-E_j)\ |\  b_0>0\},\\
&\mathcal{C}_{3, i}=\text{NM}^1(P_4)\cap\{a_i=\text{min}_{j=1}^4(a_j);  \sum_{j=1}^4a_j > a_i+a_0; \sum_{j=1}^4a_j\leq a_0+2a_i\}\\
&=\{ b_0(-K)+b_i(-K-H)+\sum_{j=1}^4 b_j(H-E_j)-b_i(H-E_i)\ |\  b_i>0\},\\
&\mathcal{C}_{4, i}=\text{NM}^1(P_4)\cap\{a_i=\text{min}_{j=1}^4(a_j);   \sum_{j=1}^4a_j>  a_0+2a_i\}=\{ b_0(-K-H)\\
&+b_i(-K-H+E_i)+\sum_{j=1}^4 b_j(H-E_j)-b_i(H-E_i)\ |\  b_i>0\}.
\end{align*} 

Clearly any two sub-cones as above do not intersect unless both are of the form $\mathcal{C}_{3, i}$.  Assume that  the set of indices $\{ i, j, k, l\}$ is exactly $\{1, 2, 3, 4\}$.  We list all possible intersection sub-cones as follows: 
 \begin{align*}
&\mathcal{C}_{3, ij}=\mathcal{C}_{3, i}\cap \mathcal{C}_{3, j} \\
& =\{ b_0(-K)+b_i(-K-H)+b_k(H-E_k)+b_l(H-E_l)\ |\ b_i>0\};  \\
&\mathcal{C}_{3, ijk}=\mathcal{C}_{3, i}\cap \mathcal{C}_{3, j}\cap\mathcal{C}_{3, k}\\
&=\{ b_0(-K)+b_i(-K-H)+b_l(H-E_l) \ |\ b_i>0 \};  \\
&\mathcal{C}_{3, ijkl}=\cap_{i=1}^4\mathcal{C}_{3, i}=\{ b_0(-K)+b_1(-K-H) \ |\   b_1>0\}.
\end{align*}
Denote by $F_{\mathcal{C}}(t)$ the sum $\sum_{ \overrightarrow{v}\in\mathcal{C}}t^{\overrightarrow{v}}$ over a  subset $\mathcal{C}$ of a lattice. 
Therefore,
 \begin{align}
&\rho_4(t_0,\cdots, t_4)=F_{\mathcal{C}_1}(t)+F_{\mathcal{C}_2}(t)+\sum_{i=1}^4(F_{\mathcal{C}_{3, i}}(t)+F_{\mathcal{C}_{4,i}}(t))\\
&-\sum_{1\leq i\neq j\leq 4}F_{ \mathcal{C}_{3, ij}}(t)+\sum_{1\leq i\neq j\neq k\leq 4}F_{\mathcal{C}_{3, ijk}}(t)-F_{\mathcal{C}_{3, ijkl}}(t) \nonumber \\
&=\left(\frac{1}{1-t_0} +\frac{t^{-K}}{1-t^{-K}}\right)\cdot e_4+\frac{t^{-K-H}\cdot  (e_3-e_2+e_1-e_0)}{(1-t^{-K})(1-t^{-K-H})}    \nonumber  \\ 
&+\frac{1}{1-t^{-K-H}} \sum_{i=1}^4\left(\frac{t^{-K-H+E_i}}{1-t^{-K-H+E_i}}\cdot \frac{1-t_0/t_i}{\prod_{j=1}^4(1-t_0/t_j)}\right),\nonumber
\end{align}
where $e_0=1$, $ e_k$ is the elementary symmetric polynomial of degree $k$ in $(x_1, \cdots, x_4)$, and $x_k=\frac{1}{1-t_0/t_k}$ for $1\leq k\leq 4$.

\end{document}